\documentclass{amsart}[12pt,a4paper]

\usepackage[applemac]{inputenc}
\usepackage[colorlinks=true,pagebackref,hyperindex]{hyperref}

\usepackage[english]{babel} 
\usepackage{amsfonts}  
\usepackage{amsmath}
\usepackage{amssymb} 
\usepackage{mathrsfs} 
\usepackage{color}  

\usepackage{amsthm}
\usepackage{latexsym}
\usepackage[T1]{fontenc} 
\usepackage{graphicx}
\usepackage{graphics} 
\usepackage{eufrak}
\usepackage{fancyhdr}
\input xy
\xyoption{all}
\usepackage{enumerate}

\newcommand{\Q}{\mathbb{Q}}
\newcommand{\Z}{\mathbb{Z}}
\newcommand{\G}{\mathbb{G}_m}
\newcommand{\C}{\mathbb{C}}
\newcommand{\F}{\mathcal{F}}
\newcommand{\De}{\Delta}
\newcommand{\g}{\gamma}

\DeclareMathOperator{\Pic}{Pic}

\DeclareMathOperator{\Mon}{Mon}
\DeclareMathOperator{\HMon}{HMon}
\DeclareMathOperator{\Aut}{Aut}
\DeclareMathOperator{\GL}{GL}
\DeclareMathOperator{\PGL}{PGL}
\DeclareMathOperator{\Nef}{Nef}
\DeclareMathOperator{\im}{im}
\DeclareMathOperator{\NE}{\overline{NE}}
\DeclareMathOperator{\Hom}{Hom}
\DeclareMathOperator{\Mov}{\overline{Mov}}

\newtheorem{theorem}{Theorem}[section]
\newtheorem{lemma}[theorem]{Lemma}
\newtheorem{proposition}[theorem]{Proposition}
\newtheorem{question}[theorem]{Question}
\newtheorem{example}[theorem]{Example}
\newtheorem{corollary}[theorem]{Corollary}
\newtheorem{con}[theorem]{Conjecture}
\newtheorem{rem}[theorem]{Remark}
\newtheorem{remark}[theorem]{Remark}

\theoremstyle{definition}

\newtheorem{definition}[theorem]{Definition}

\theoremstyle{remark}

\newcommand{\N}{\mathbb{N}}
\newcommand{\rr}{\mathbb{R}}
\newcommand{\qq}{\mathbb{Q}}
\newcommand{\zz}{\mathbb{Z}}

\newcommand{\pp}{\mathbb{P}}

\newcommand{\oo}{\mathcal{O}}

\begin{document}
\title{Fano varieties in Mori fibre spaces}

\author[G. Codogni]{Giulio Codogni}
\address{Dipartimento di Matematica e Fisica, Universit\`a Roma Tre, Largo San Leonardo Murialdo, 1
00146, Roma, Italy.}
\email{codogni@mat.uniroma3.it}

\author[A. Fanelli]{Andrea Fanelli}
\address{Department of Mathematics, Imperial College London, 180 Queen's Gate,
London SW7 2AZ, UK.}
\email{a.fanelli11@imperial.ac.uk}

\author[R. Svaldi]{Roberto Svaldi}
\address{Department of Mathematics, Massachusetts Institute of Technology, 50 Ames Street, Cambridge, MA, 02142, USA.}
\email{rsvaldi@math.mit.edu}

\author[L. Tasin]{Luca Tasin}
\address{Mathematical Institute of the University of Bonn, Endenicher Allee 60
D-53115 Bonn, Germany.} 
\email{tasin@math.uni-bonn.de}


\begin{abstract}
We show that being a general fibre of a Mori fibre space is a rather restrictive condition for a Fano variety. More specifically, we obtain two criteria (one sufficient and one necessary) for a $\qq$-factorial Fano variety with terminal singularities to be realised as a fibre of a Mori fibre space, which turn into a characterisation in the rigid case. We apply our criteria to figure out this property up to dimension three and on rational homogeneous spaces. The smooth toric case is studied and an interesting connection with $K$-semistability is also investigated.\end{abstract}

\maketitle

\tableofcontents

\section{Introduction}

In this work we focus on a natural question which arises in the context of the classification of complex algebraic varieties and in the minimal model program (or MMP) and tries to clarify the geography of Mori fibre spaces.

\begin{question}\label{generalqn}
Which $\qq$-factorial Fano varieties can be realised as general fibres of a Mori fibre space?
\end{question}

Although every Fano variety of Picard number one is a Mori fibre space over a point, this work gives evidence about the restrictiveness of this condition for varieties of higher Picard rank.

The notion of ``general fibre'' will be clarified later in Section \ref{sec_monodromy2}: the idea is to determine an open dense subset of the base, on which the fibres are ``good enough'' (cf. Definition \ref{fibre_like}).

Fano varieties play an essential role in the birational classification of projective varieties with negative Kodaira dimension. Their importance was already highlighted in low dimension in \cite{mori3folds}. The seminal work \cite{bchm} shows that every $\mathbb{Q}$-factorial variety with klt singularities and 
non-pseudoeffective canonical divisor is birational to a Mori fibre space (or simply MFS), i.e., to a contraction morphism with positive dimensional Fano fibres and relative Picard number one. In this work, we will also assume the existence of a dense open set of the base over which the fibres are $\mathbb{Q}$-factorial.

Since Mori fibre spaces arise as final outcomes of a run of the MMP, they have been widely studied for the last thirty years in the context of classification of higher dimensional varieties.

It is important to underline that distinct Mori fibre spaces can belong to the same birational class, as shown already in dimension 2 by elementary transformations between ruled surfaces. Relations between Mori fibre spaces within a birational class (the so-called {\it Sarkisov program}) were investigated by \cite{cortisar} in low dimension. The same picture has been proved to endure in higher dimension in \cite{hacmcsar}: two Mori fibre spaces within the same birational class can be related via a sequence of very easy birational maps, called {\it Sarkisov links}. Another interesting notion for Mori fibre spaces which appears in the literature is birational rigidity (cf. \cite{brcor}). Although so many properties have been investigated, the geometric structure of Mori fibre spaces remains quite mysterious and very few explicit examples are known.
In this work we focus on the classification of the fibres of MFS's rather than the total space. 

The main results of this paper are the following criteria (cf. Theorem \ref{sufficientCriterion}, Theorem \ref{necessaryCriterion} and Theorem \ref{rigidchar}). We will denote by $\Mon(F)$ the maximal subgroup of $\GL(N^1(F),\zz)$ which preserves the birational data of $F$ (cf. Definition \ref{MonGroup}) and by $\Aut(F)$ the image of the natural homomorphism of the automorphism group of $F$ with value in $\GL(N^1(F),\zz)$. Moreover, for a given group $G$ acting on $N^1(F)_\qq$, we will denote by $N^1(F)^G_\qq$ the $G$-invariant subspace, i.e. the subspace of classes $v \in N^1(F)_\qq$ such that $gv=v, \; \forall g \in G$.

\begin{theorem}\label{mainthm}
\
\begin{itemize}
\item {\bf Sufficient criterion:}  A terminal $\mathbb{Q}$-factorial Fano variety $F$ can be realised as the general fibre of a MFS if
\[
N^1(F)^{\Aut(F)}_\qq = \qq K_F.
\]

\item { \bf Necessary criterion:} A terminal $\mathbb{Q}$-factorial Fano variety $F$ such that
$$\dim N^1(F)_\mathbb{Q}^{\Mon(F)} > 1$$
cannot be realised as the general fibre of a MFS.
\item {\bf Characterisation for rigid varieties:} Assume that $H^1(F, T_F)= 0$, i.e., $F$ is rigid. Then the sufficient criterion turns into a characterisation. 
\end{itemize}
\end{theorem}
Our results rely upon a careful study of the monodromy action on Mori fibre spaces.

\begin{remark}
The notion of being realised as the general fibre of a MFS (or {\it fibre-likeness}, see Definition \ref{fibre_like}) is quite subtle and it is important to remark that the necessary criterion holds true also for Mori fibre spaces which are not isotrivial.
\end{remark}

We can use our criteria to prove the following theorem (cf. Theorems \ref{mori_surf} and \ref{3folds.MFS.thm}, Corollaries \ref{Cor_Surf} and \ref{Cor_3fold}).

\begin{theorem}\label{surf_three}\hfill
\begin{itemize}
\item {\bf Surfaces:} A smooth del Pezzo surface can be realised as the general fibre of a MFS if and only if it is not isomorphic to the blow-up of $\pp^2$ in one or two points.
\item {\bf Threefolds:} The deformation type of a smooth Fano threefold $F$ with $\rho(F) >1$ can be realised as the general fibre of a MFS if and only if it is one of the 8 classes appearing in Table \ref{3folds.MFS.table}.
\end{itemize}
In particular fibre-likeness is invariant under smooth deformations for Fano varieties of dimension up to three; moreover, in these cases the necessary criterion of Theorem \ref{mainthm} is actually a characterisation.
\end{theorem}

Let us point out that cubic surfaces in $\pp^3$ provide examples of varieties that do {\it not} satisfy our sufficient criterion but can be realised as a general fibre of a MFS. To deal with Fano threefolds in Section \ref{sec_3fold}, we give some ad hoc versions of the necessary criterion explicitly in terms of the birational geometry of $F$. We do not know if there are higher dimensional examples of smooth Fano varieties which are not fibre-like but still verify the necessary criterion of Theorem \ref{mainthm}.

\begin{rem}
The 2-dimensional case of Theorem \ref{surf_three} has been worked out in \cite[Theorem 3.5]{Mori_surf} when the dimension of the total space of the Mori fibre space is three. In Section \ref{sec_3fold} we give an alternative proof using our criteria, allowing total spaces of arbitrary dimension.

Mori and Mukai classified all smooth Fano threefolds with Picard number bigger than one up to deformation into 88 classes in \cite{mori.mukai.class} and \cite{mori.mukai.class.2}: Theorem \ref{surf_three} shows how restrictive the fibre-likeness condition is.

The second part of Theorem \ref{surf_three} can be deduced by combining our criteria with the classification result of Prokhorov, see \cite[Theorem 1.2]{prok.Gfano.2}. We prove the theorem without using Prokhorov's work, but looking directly at the nef cone of $F$.
\end{rem}

On the positive side, we can give the following examples of varieties of higher dimension and large Picard number which can be realised as the general fibre of a Mori fibre space.
\begin{definition}
Let $Z$ be a variety and  $L_1, \dots, L_k$ Cartier divisors on $Z$. 

Let $I$ be the subvariety of $Z\times |L_1| \times \cdots |L_k|$
such that for any divisors $D_1 \in |L_1|, \dots, D_k \in |L_k|$,
the fiber of the projection morphism $I \to |L_1| \times \cdots |L_k|$
above $(D_1, \dots, D_k)$ equals the complete intersection $D_1 \cap \dots \cap D_k$ in $Z$.

We call $I$ the incidence variety.
\end{definition}

\begin{theorem}[= Theorem \ref{GeneralConstruction}]
Let $F$ be a smooth Fano variety. Let $Z$ be a smooth projective variety. Assume that $L_i, \; 1 \geq i \geq k$ are basepoint-free effective Cartier divisors on $Z$ and that $F$ is the complete intersection of the divisors $L_i$ in the ambient variety $Z$. 

Let $I$ the incidence variety in $Z\times |L_1| \times \cdots |L_k|$. Suppose that there is a finite cyclic subgroup $G$ of $\Aut(Z)$ which is fixed point free in codimension one and whose action can be lifted to $I$. Assume that G does not preserve $F$. If 
$$  \dim N^1(Z)_\mathbb{Q}^G=1,  $$
then $F$ is fibre-like.
\end{theorem}
We will give a few concrete applications of this result; let us point out one of them.
\begin{corollary}[= Corollary \ref{esempi_in_dimensione_alta}]
Take positive integers $r,k,d,$ and $n\geq 2$ such that $kd<n+1$. Let $F$ be a smooth complete intersection of $k$ divisors of degree $(d,\dots ,d)$ in $(\pp^n)^{r}$. Then $F$ is fibre-like.
\end{corollary}

In the case of smooth toric Fano varieties, we obtain the following necessary condition.

\begin{theorem}[= Thm \ref{cor_kstability}]
Let $F$ be a toric Fano variety and let $\Sigma\subset N$ be the fan associated to it.
Let $\Delta$ be the polytope whose vertices are the integral generators of the $1$-dimensional cones of 
$\Sigma$.

If $F$ can be realised as the general fibre of a MFS then the barycentre of $\Delta$ is in the origin.
\end{theorem}

\begin{remark}
Smooth fibre-like toric varieties have been completely classified in low dimension: as Table \ref{t1} shows, it is a rather restrictive condition. 
\end{remark}

In section \ref{sec_bandiere} we classify fibre-like rational homogeneous spaces.

We do not know if the fibre-likeness (i.e. the property of being realised as the general fibre of a MFS) is open in families. Our necessary criterion in Theorem \ref{mainthm} is invariant under flat deformation, while the sufficient criterion is closed in families, but it detects just a special kind of MFS's: the {\it isotrivial} ones. 

In the case of del Pezzo surfaces and of toric varieties we are able to draw an interesting connection between fibre-likeness and $K$-stability for smooth Fano varieties. So far, we can prove the following result (cf. Corollary \ref{delpezzo_stab} and Theorem \ref{cor_kstability}).

\begin{corollary}
\hfill
\begin{itemize}
\item A smooth del Pezzo surface is $K$-stable if and only if it can be realised as a general fire of a MFS.
\item If a smooth toric Fano variety $F$ appears as the general fibre of a Mori fibre space, then it is $K$-stable.
\end{itemize}
\end{corollary}

Furthermore, fibre-like smooth threefolds are suspected to be $K$-stable (cf. Remark \ref{rui_st}).
Inspired by \cite{OO13}, we propose the following question, which is the relative version of a conjecture by Odaka and Okada.

\begin{question}
Is it true that every smooth Fano variety which can be realised as a general fibre of a MFS is $K$-semistable?
\end{question}

\section*{Acknowledgements}
This project started during the summer school PRAGMATIC 2013 in Catania, we thank the organisers Alfio Ragusa, Francesco Russo and Giuseppe Zappal\`a for their kind hospitality, the stimulating research environment and the financial support during the school.

The problem was proposed by Paolo Cascini and Yoshinori Gongyo. We kindly thank them for their help and support.

We are grateful to Yujiro Kawamata for his lectures in Catania and his encouragement.

Furthermore, we had benefit from conversations with Cinzia Casagrande, Alessio Corti, Tommaso de Fernex, Ruadha\'i Dervan, Gabriele Di Cerbo, Enrica Floris, Jep Gambardella, Al Kasprzyk, Angelo Felice Lopez, James M\textsuperscript{c}Kernan, Yuji Odaka, Marco Pizzato, Roberto Pirisi, Edoardo Sernesi, Nick Shepherd-Barron, Alessandro Verra, Angelo Vistoli, Filippo Viviani. We heartily thank them.

We would like to thank Kento Fujita for his precious help to complete the classification of smooth fibre-like threefolds and the two referees for their helpful comments.

G.C. is funded by the grant FIRB 2012 "Moduli Spaces and Their Applications". A.F. is funded by the grant ``Roth Scholarship'' of the Department of Mathematics at Imperial College London. R.S. was hosted by the Department of Mathematics of UC San Diego, during part of the development of this project. He wishes to thank UCSD for the kind hospitality and the nice working environment. He is partially supported by NSF DMS-0701101 and DMS-1200656. L.T. is partially supported by GNSAGA-INdAM  and was funded by the ``Max Planck Institute for Mathematics'' of Bonn during part of the development of this project.

The final revision of this paper was completed while the first and the second author were visiting IMPA for the ``Thematic Program on Algebraic Geometry 2015''.

\section*{Notation}

By a variety $X$ we mean an integral separated scheme of finite type over a field $k$. Unless otherwise specified we work over the field of complex numbers. 

We follow the usual notation from \cite{SingularityMMP} and \cite{bchm}. In particular, the reader can refer to these two references for the general definitions about the singularities of the Minimal Model Program.

We denote with $N^1(X)_{\qq}$ the group of Cartier divisors of $X$ modulo numerical equivalence, after tensoring by $\qq$. We denote with $N_1(X)_{\qq}$ the group of curves of $X$ modulo numerical equivalence, after tensoring by $\qq$. By construction, these two rational vector spaces are dual.

For a projective morphism of varieties $f \colon X \to Y$, we denote by $N_1(X/Y)_{\qq}$ the subspace of $N_1(X)_\qq$ generated by curves that are mapped by $f$ to a point. The group $N^1(X/Y)_\qq$ is the dual of $N_1(X/Y)_\qq$.

\begin{section}{Preliminary results}\label{sec_monodromy2}

\begin{subsection}{Monodromy Action and Deligne's Theorem}
We will discuss some facts about monodromy on fibrations in Fano varieties. 
Let us recall a basic definition. 

\begin{definition}
A normal projective variety $F$ is said to be {\it Fano} if its anti-canonical divisor $-K_F$ is $\mathbb{Q}$-Cartier and ample, i.e., there exists a sufficiently high multiple $-nK_F$ which is Cartier and the invertible sheaf $\mathcal{O}_F(-nK_F)$ is very ample.
\end{definition}

In this subsection, we deal with projective morphisms between normal varieties
\[f\colon X \rightarrow Y\]
such that $f_\ast \mathcal{O}_X =\mathcal{O}_Y$, i.e., the fibres are connected.
Moreover we assume that $X$ is $\mathbb{Q}$-factorial with rational singularities and $-K_X$ is relatively ample over $Y$. 
The assumption on the ampleness of $-K_X$ implies that the general fibre is a Fano variety.
As our final goal is to discuss a classification of fibres of Mori fibre spaces, the reader will notice that towards the end of this section (cf. Definition \ref{mfs.def}) we will make further assumptions on the singularities of $X$; for the time being, though, we will try to keep our treatment to the widest possible level of generality.

We consider two sheaves on $Y$: the first one is $R^2f_*\Q$, 
whose fibre at $t$ is $H^2(F_t,\Q)$; 
the second one is $R^1f_*\G\otimes \Q$, whose fibre at $t$ is $\Pic(F_t)_{\Q}$. The first Chern class defines a morphism
\begin{equation}\label{firstChern}
c_1\colon R^1f_*\G\otimes \Q \to R^2f_*\Q.
\end{equation}
When $X$ is mildly singular, for example when $X$ has Kawamata log terminal (in short, {\it klt}) singularities, the above 
assumptions imply that the map $c_1$ is an isomorphism, as the relative version of Kawamata-Viehweg vanishing applies. 

When both $X$ and $Y$ are smooth, by Sard's Theorem there is an open dense subset $U$ of $Y$ 
where $f$ is a submersion. 
Using Ehresmann's Theorem (cf. \cite[Proposition 9.3]{Voisen}), it follows that $f$ is a locally trivial 
fibration of topological spaces over $U$. 
The existence of a non-empty Zariski open subset $U_f^{top}$ of $Y$ where $f$ is a locally trivial topological 
fibration holds also in the singular case by a delicate argument due to Verdier (cf. \cite[Corollary 5.1]{Verdier}). 
In our exposition, we will denote this open set by $U^{top}$.
Unfortunately, the characterisation of $U^{top}$ is not easily obtained and leads to the concept of equisingularity (cf. \cite{Teissier}). 

On $U^{top}$, the sheaf $R^2f_*\Q$ is a local system. 
Thus, there is a \emph{monodromy action} of $\pi_1(U^{top},t)$ on the fibre $(R^2f_*\qq)_t$. 
In this set-up, we can obtain a more refined result.

\begin{theorem}[\cite{kollarmori}, \cite{DH}]\label{TopTriv}
Let $U=U_f$ be the open subset, possibly empty, of $Y$ where 
\begin{enumerate}
\item $Y$ is smooth;
\item $f$ is flat;
\item the fibres of $f$ are $\Q$-factorial with terminal singularities.
\end{enumerate}
Then the sheaf $R^1f_*\G\otimes \Q$ is a local system on $U$ with finite monodromy.
\end{theorem}

\begin{proof}
Since the fibres of $f$ over $U$ are terminal with $\mathbb{Q}$-factorial singularities they satisfy \cite[Conditions 12.2.1]{kollarmori} as explained in \cite[Remark 12.2.1.4]{kollarmori}.
Hence we can define the sheaf $\mathcal{GN}^1(X/U)$ as in \cite[Def. 12.2.4]{kollarmori} and show that it is a local system with finite monodromy isomorphic to $R^1f_*\G\otimes \Q$ at a very general point of $U$. The following step is taken from \cite[Prop. 6.5]{DH}. The authors first show that $\mathcal{GN}^1(X/U)$ is isomorphic to $R^1f_*\G\otimes \Q$ at the general point of $U$; that essentially follows from Verdier's result and the isomorphism (\ref{firstChern}). Moreover, they show that actually the isomorphism holds at every point of $U_f$. 

We remark that in \cite{DH} the base is a smooth curve and the Fano varieties involved have terminal singularities. Nonetheless, the same argument applies verbatim to smooth bases of arbitrary dimension.
\end{proof}

\begin{remark}\label{q-fact.rmk}
Assumption 3. in the statement of the theorem is needed to define the sheaf $\mathcal{GN}^1(X/U)$.
By the discussion in \cite[Remark 12.2.1.4]{kollarmori} this hypothesis may be weakened to the requirement of fibres having $\Q$-factorial singularities and being smooth in codimension 2.

In general, the above set $U_f$ may be empty. Nonetheless, the reader should keep in mind at this stage 
that our goal is to classify fibres of a Mori fibration (cf. Definition \ref{mfs.def}). In particular, in the following sections we will often be able to assume that either the singularities of $X$ are terminal, hence the general fibre of $f$ will also be terminal, or by construction the fibres of $f$ will have terminal singularities.

Finally, let us comment on the $\Q$-factoriality of the fibres. In \cite[Thm. 12.1.10]{kollarmori}, the authors show that in a flat family $T \to S$ of varieties with rational singularities and smooth in codimension $2$, there is an open subscheme $W \subset S$ parametrizing $\Q$-factorial fibres. These hypotheses are clearly satisfied in the case of a family of Fano varieties with terminal singularities or with klt singularities and smooth in codimension 2.
The problem is that even assuming that $T$ is $\Q$-factorial, then the set $W$ may be empty as the following example shows. 
\end{remark}
\begin{example}
Let $C$ be a projective curve of genus $g \geq 1$ and let $C'$ be a degree $2$ \'etale cover of $C$ and call $i$ the associated involution.
Let $Q$ be a quadric in $\mathbb{P}^4$ of rank 3, i.e., the projective cone of $\mathbb{P}^1 \times \mathbb{P}^1$ for the Segre embedding. The action of $\Z/2\Z$ switching the factors on $\mathbb{P}^1 \times \mathbb{P}^1$ lifts to an automorphism of $Q$, which we denote by $g$. Define $T$ to be the quotient of $Q \times C'$ by the involution $(g, i)$. Then $T$ maps to $C$ and the fibres are all isomorphic to $Q$. The class group of $Q \times C'$
has rank 3 as it is isomorphic to the direct sum of the class groups of $Q$ and $C'$. On $T$, the monodromy induced by the quotient reduces the rank of the class group to $2$. The Picard group of $T$ instead has rank $2$ as the Picard group of $Q \times C'$. We conclude that the morphism $T\to C$ is an isotrivial fibration of relative Picard number $1$ such that $T$ is $\mathbb{Q}$-factorial but none of the fibres is $\mathbb{Q}$-factorial.
\end{example}

On $U^{top}$, the monodromy action on $R^2f_*\Q$ can be explicitly described as follows. Take the class of a loop $\g$ in $\pi_1(U,t)$. Pull back $X$ to the interval $I=[0,1]$ via $\g$. Trivialise the family $X_I$. The identification between the fibres over the endpoints of the interval $I$ is a homeomorphism of $F_t$: such homeomorphism gives the monodromy action. If we change the representative of $[\g]$ we may change the automorphism, but not its action on $H^2(F_t,\qq)$. See \cite[Section 9.2.1, 15.1.1 and 15.1.2]{Voisen} for more details. 

In general, we will consider the monodromy action of the fundamental group of different open subsets of $Y$. However, if we have a normal variety $W$ with a closed subvariety $Z$, the natural morphism
$$\pi_1(W\setminus Z)\to \pi_1(W) $$
is always surjective (cf. \cite[0.7 (B)]{FL}). As we will be investigating Mori fibre spaces (see Definition \ref{mfs.def}), restricting to open subsets of $U^{top}$ is not going to change the monodromy action very much and in particular it will not affect the dimension of the invariant part that is of dimension $1$.

Restricting a cohomology class to each fibre we have a morphism
\[H^2(X,\qq)\rightarrow H^0(U,R^2f_*\qq).\]
By evaluating the section at $t$ we get an isomorphism
\[ H^0(U,R^2f_*\qq)\rightarrow H^2(F_t,\qq)^{\pi_1(U,t)}.\]
Composing the above maps we obtain a morphism
\[\rho\colon H^2(X,\Q)\rightarrow H^2(F_t,\Q)^{\pi_1(U,t)}.\]
When $X$ and $Y$ are smooth, a result due to Deligne (cf. \cite[Theorem 16.24]{Voisen}) states that $\rho$ is surjective for every $t \in U$, providing a method to easily identify the invariant part of cohomology. Deligne's result is rather general and it concerns all cohomology groups and cohomology classes which are not, in general, algebraic. However, in our case, the cohomology classes we are interested in are just the classes of divisors, so we can generalise Deligne's result to the singular case for $N^1_\Q$.
\begin{theorem}\label{CorGen}
Let 
\[f\colon X\rightarrow Y\]
be a dominant morphism of projective normal varieties, where $X$ is $\qq$-factorial with rational singularities. Assume that $-K_X$ is $f$-ample. Take $U=U_f$ as in Theorem \ref{TopTriv} and assume it contains a point $t$. Then the restriction map
\[\rho \colon N^1(X)_{\Q}\rightarrow N^1(F_t)_{\Q}^{\pi_1(U,t)}\]
is surjective for every $t$ in $U$.
\end{theorem}
\begin{proof}
By Theorem \ref{TopTriv} and \cite[Corollary 12.2.9]{kollarmori}, $\mathcal{GN}^1(f^{-1}(U_f)/U_f)$ is a local system on $U_f$ whose global sections are
$N^1(f^{-1}(U_f)/U_f)_\qq$. By general properties of local systems and the exponential sequence (cf. \cite[Lemme 16.17]{Voisen}) this vector space can be identified with $H^2(F_t, \qq)^{\pi_1(U_f, t)}$ which is isomorphic to $N^1(F_t)_{\Q}^{\pi_1(U, t)}$, since $F_t$ is Fano with terminal singularities and by \cite[0.7 (B)]{FL}.
Hence, 
\[
N^1(f^{-1}(U_f)/U_f)_\qq \cong N^1(F_t)_{\Q}^{\pi_1(U, t)}
\]
where the isomorphism is given by the restriction to $F_t$. By definition of $N^1(f^{-1}(U_f)/U_f)_\qq$, 
there is a surjection 
\[
N^1(f^{-1}(U_f))_\qq \twoheadrightarrow  N^1(F_t)_{\Q}^{\pi_1(U, t)}
\]
again given by restricting to $F_t$.

The $\Q$-factoriality of $X$ implies that the surjectivity of the restriction map extends also to $N^1(X)_\qq$, proving the statement of the theorem.
\end{proof}
\end{subsection}

\begin{subsection}{The monodromy action and the MMP}
In this section we show that the monodromy action preserves some information about the birational geometry of terminal $\mathbb{Q}$-factorial Fano varieties (a general reference on this topic is \cite{dFH12}). As in the previous section we think that $F_t$ is a Fano variety which appears as a fiber in a given morphism $f \colon X \to Y$ with $-K_X$ relatively ample over the open set $U_f$ defined in Theorem \ref{TopTriv}.

First of all, the monodromy preserves the intersection pairing. Indeed, it can be seen as an action on the cohomology algebra $H^*(F_t,\Z)$. The class of the canonical divisor of $F_t$ is, by adjunction, the restriction $K_X|_{F_t}$; hence it is preserved by the monodromy. Call $n$ the dimension of $F_t$. The $\Q$-valued bilinear form
\[b(A,B):=(K_{F_t})^{n-2}\cdot A \cdot B\]
on $N^1(F_t)_\qq$ is equally preserved. 

What else is preserved?

When all fibres are smooth Fano varieties of the same dimension, Wi\'{s}niewski proved in \cite{Wis1} and \cite{Wis2} that the nef cone is locally constant.  In the terminal $\qq$-factorial case, \cite[Theorem 6.8]{DH} shows that the movable and pseudoeffective cones are preserved.
As $F_t$ is Fano, it follows that it is also a Mori dream space, \cite[Cor. 1.3.2]{bchm}. In particular, the movable cone $\Mov(F_t)$ admits a finite decomposition into polyhedral cones, called Mori chambers decomposition, \cite[Prop. 1.11]{hu.keel}. The cone $\Nef(F)$ is one of the chambers in the decomposition.

In \cite[Theorem 6.9]{DH}, the authors shows that the Mori chambers are locally constant in the following cases:
\begin{itemize}
\item families of $3$-dimensional Fano varieties;
\item families of $4$-dimensional and $1$-canonical Fano varieties;
\item families of toric Fano varieties.
\end{itemize}
In these cases the nef cone of fibres is locally constant in the family. In \cite{Totaro}, Totaro has shown that under weaker assumptions than the one just illustrated the Mori chambers are not preserved by the monodromy action.

We can prove the following theorem.
 \begin{theorem}\label{Wis_nef}
Keep notation as in Theorem \ref{TopTriv}. Up to shrinking $U_f$ there exists an open set $U'_f \subset U_f \subset Y$ on which the monodromy action preserves $\Nef(F_t)$. 
\end{theorem}
\begin{proof}
By Theorem \ref{TopTriv}, there is a finite \'{e}tale cover $p\colon V\rightarrow U_f$ trivialising the monodromy action. Let us denote by $X_V:=V \times_{U_f} X$ the pull-back of the family of Fano varieties over $U_f$ to $V$ with $f_p\colon X_V \to V$ the associated map.

Denote by $N_1(X_V/V)$ the quotient $N_1(X_V)/f_p^\ast N_1(V)$. For every fibre $F_t$ of $f_V$, the restriction map $N_1(X_V/V)_{\mathbb{Q} }\to N_1(F_t)_{\mathbb{Q}}$ is an isomorphism.

Since $-K_X$ is $f$-ample, by the cone theorem we know that $\NE(X_V/V)$ - the cone generated by effective classes of curves - is rational polyhedral: in particular it is generated by a finite number of rays $R_1, \ldots, R_k$ and each ray is generated by the class of a curve $C_i$. Any $C_i$ is represented by an integral lattice point in $N_1(X_V/V)_{\qq}$.

By Kawamata's rationality theorem, \cite[Thm. 3.5]{SingularityMMP}, the primitive generators of $\NE(X_V/V)$ are integral lattice points that lie between the hyperplanes $H_0=\{v \in N_1(X/Y) \ | \ K_X\cdot v =0\}$ and $H_{2n}=\{v \in N_1(X/Y) \ | \ K_X\cdot v =2n\}$, where $n=\dim X$. The number of integral points contained in the closure of the relative cone of effective curves, $\NE(X_V/V)$, laying between $H_0$ and $H_{2n}$ is of course finite. Let us denote the set of such points by $C$.

Now, to $C$ we can associate the variety 
\[
M:=\bigcup_{\beta \in C} \mathrm{Mor}(\mathbb{P}^1, X_V/V, \beta)
\] 
of morphisms from $\mathbb{P}^1$ to $X_V$, contracted by $X_V\rightarrow V$, whose images have class belonging to $C$. This is a quasi-projective scheme of finite type that comes equipped with a proper map to $\pi: M \to V$. Let $N$ be the union of those irreducible components of $M$ that do not dominate $V$ via $\pi$, and let $T=\pi(N)$. Let us remark that $T$ is a Zariski closed set of $V$, as $M$ is a variety of finite type proper over $V$.

Define $W:= V \setminus T$ and $X_W= f^{-1}(W)$. The claim is that $p(W)$ is the Zariski open set we are looking for. In fact, consider the rational polyhedral cone $\NE(X_W/W)$. As above, the extremal rays are finite and their generators have bounded degree with respect to $-K_X$. As every component of $M$ that does not belong to $N$ dominates $V$, then the classes that generate the extremal rays of $\NE(X_W/W)$ move over $W$. But this means that the restriction $N_1(X_W/W)_{\qq} \to N_1(F_t)_{\qq}$ identifies $\NE(X_W/W)$ and $\NE(F_t)$. The inclusion $\NE(F_t) \subset \NE(X_W/W)$ follows in fact directly from the definition, while the opposite one is a consequence of the last observation.
\end{proof}
Our result does not provide any effective method to characterise the open subset where the monodromy preserves the nef cone. 

To ease the notation, from now on we will denote by $U$ the open subset constructed in Theorem \ref{Wis_nef}. \newline
Take an element $g$ of $\pi_1(U,t)$. Assume it exchanges two maximal faces $
\mathcal{G}_1$ and $\mathcal{G}_2$ of the nef cone. These faces give contractions 
\[\pi_i\colon F_t \rightarrow G_i,\]
which correspond to the first step in a run of the $K_{F_t}$-MMP.  
The pull-back via $\pi_i$ identifies $\mathrm{Nef}(G_i)$ with the 
face $\mathcal{G}_i$. The first step of such a run is one of the following three possibilities.
\begin{description}
\item[Divisorial contraction] $\pi: F_t \to G$ is birational, the exceptional locus is an irreducible divisor 
and all the curves in the fibres are numerically equivalent, i.e. $\rho(F_t/G)=1$;
\item[Flipping contraction] $\pi: F_t \to G$ is birational, the exceptional locus is of codimension at least $2$ and 
$\rho(F_t/G)=1$;
\item[Mori fibre contraction] $\pi: F_t \to G, \; \dim F_t > \dim G$ and $\rho(F_t/G)=1$.
\end{description} 
The information whether the map $\pi_i$ is a divisorial contraction or a flipping contraction or a Mori fibration is encoded in the cohomology ring. 
As a consequence, each of the above types is preserved under monodromy.
\begin{theorem}\label{Monodromy_And_MMP}
Using the same notations as above. Let us consider the flat family of terminal Fano's over the open set $U:=U'_f$ defined in Theorem \ref{Wis_nef}. Let $t \in U$ be a point and let $F_t$ be the fibre over $t$. Assume that the monodromy action identifies two, not necessarily maximal, faces 
$\mathcal{G}_1$ and $\mathcal{G}_2$ of the nef cone of $F_t$. Then, the two maps
\[\pi_1 \colon F_t \to G_1\]
\[\pi_2 \colon F_t \to G_2\]
correspond to the same kind of step in the MMP. In the case of the divisorial contraction, the monodromy action exchanges the exceptional divisors.

Moreover, the varieties $G_1$ and $G_2$ (and the morphisms $\pi_1$ and $\pi_2$) are deformation equivalent.
\end{theorem}

\begin{proof}
To ease the notation throughout the proof we will indicate $F_t$ simply by $F$.

To prove the first part of the theorem, we do a case-by-case analysis.
\begin{description}
	\item[Divisorial contraction] As $\pi_i \colon F_t \to G_i$ is birational, given a divisor $H$ in the relative interior of $\pi_i^\ast\Nef(G_i)$, we have $(H^{\dim F_t}) >0$.  The exceptional locus is an irreducible divisor, call it $D_i$. It is clear that $D_i$ is the only 
effective divisor on $F$ such that $(H^{\dim F_t-1}\cdot D_i)= 0$, for every $H$ in the relative interior of the 
corresponding face. Moreover, we can characterise the dimension of $\pi_i(D_i)$ as the maximal integer $k$ 
such that $(H^{k}\cdot D_i) \neq 0$, i.e., the numerical dimension of the restriction of $H$ to $D_i$.

\item[Flipping contraction] as $\pi_i \colon F_t \to G_i$ is birational, given a divisor $H$ in the relative interior of $\pi_i^\ast\Nef(G_i)$, we have 
$(H^{\dim F_t}) \neq 0$. The smallness of $\pi_i$ is equivalent to the fact that for every effective divisor $E \in \mathcal{G}_i$ we have 
$(H^{\dim F_t-1}\cdot E) > 0$. Both these conditions are preserved by the monodromy action (as the effective cone is 
preserved).

\item[Mori fibre contraction] as $\dim F_t > \dim G_i$, given a divisor $H$ in the relative interior of 
$\pi_i^\ast\Nef(G_i)$, we have $(H^{\dim F_t}) = 0$ and $\dim G_i$ is the maximum integer such that $(H^k) \neq 0$.
Hence, in this case, we also know that the dimension of the base of the fibration is preserved by monodromy.
\end{description} 

Let us prove that there exists a flat deformation from $G_1$ to $G_2$. The monodromy action is finite, so after a finite \'{e}tale cover $p\colon V \to U$, we obtain a family 
\[
f_p\colon X_V \to V
\]
 with trivial monodromy action on the fibres. This means that the restriction morphism 
 $N^1(X_V)_\qq \to N^1(F)_\qq$ is surjective for every fibre of $f_p$. 
As $\mathcal{G}_1$ and $\mathcal{G}_2$ are two faces identified under the 
monodromy action on the family $X_U \to U$, we can find two point $t_1, t_2$ on $V$ 
and a Cartier divisor $H$ on $X_V$ whose 
restrictions to $F_{t_1}$ (resp. $F_{t_2}$) lie in the relative interior of $\mathcal{G}_1$ (resp. $\mathcal{G}_2$).

We want to construct the variety 
$\tilde{X} := {\rm Proj}_{\mathcal{O}_{V}}(\bigoplus_{n \in \N}{f_p}_\ast \mathcal{O}_{X_V}(nH))$,
getting a morphism relative to $V$
\begin{displaymath}
 \xymatrix {
  X_V \ar[rr]^{g} \ar[dr]_{f_p} && \tilde{X} \ar[dl]^{\pi} \\
  & V
}
\end{displaymath}

The fibre of $\pi$ over $t_i$ is $G_i$. The restriction of $g$ over $t_i$ is the contraction given by the face
$\mathcal{G}_i$. We are going to prove that $\tilde{X}$ and $\pi$ exist and they are flat over $V$.
Since $H$ is Cartier and $f_p$ is flat, the sheaves $\mathcal{O}_{X_V}(nH)$ are flat over $V$.
Moreover, as any fibre $F_s$ is terminal $H^i(F_s, \mathcal{O}_{F_s}(nH\vert_{F_s}))=0, i > 0, n >0$: we have assumed that $H\vert_{F_1}$ is nef and since the monodromy action on the family $X_U\rightarrow U$ preserves the nef cone, we have that $H\vert_{F_s}$ is nef for every $s \in V$.
So by the classical theory of cohomology and base change, \cite[Cor. 2]{Mumford}, the sheaves $\pi_\ast\mathcal{O}_{X_V}(nH)$ are locally free sheaves. Hence $\bigoplus_{n \in \N}{f_p}_\ast \mathcal{O}_{X_V}(nH)$ is a flat sheaf of algebras. The finite generation now follows from Castelnuovo-Mumford regularity, up to passing to a sufficiently large multiple of $H$, cf. \cite[Example 1.8.24]{laz1.book}.

Hence $\pi$ is flat and hence $G_1$ is deformation equivalent to $G_2$, i.e.,
we have an actual flat deformation of the contraction given by $\mathcal{G}_1$ to the contraction given by $\mathcal{G}_2$. 
\end{proof}

The previous discussion motivates the following definition.
\begin{definition}[The groups $\HMon$ and $\Mon$]\label{MonGroup}
Let $F$ be a normal $n$-dimensional Fano variety with terminal singularities.
We denote by $\Aut(F)^0$ the largest subgroup of $\Aut(F)$ which acts trivially on the N\'eron-Severi group with $\Q$-coefficients. Then the group $\HMon(F)$ is defined as 
\[
\HMon(F)\colon=\Aut(F)/\Aut(F)^0.
\]
The group $\Mon(F)$ is defined as the maximal subgroup of $\GL(N^1(F),\Z)$ which preserves:
\begin{itemize}
 \item the line spanned by the class of $K_F$;
 \item the $\Q$-valued bilinear form $\,b(A,B):=(K_F)^{n-2}\cdot A \cdot B$;
 \item the nef cone;
 \item the type of step (divisorial, flipping, fibre type) of the MMP associated to the facets of the nef cone and the exceptional divisor;
 \item the deformation type of the images of the maps defined by the faces of the nef cone.
 \end{itemize}
\end{definition}

\begin{rem}
The group $\HMon(F)$ is the subgroup defined as the image of the natural homomorphism
\[
\Aut(F) \to \Mon(F) \subseteq \GL(N^1(F),Z)
\]
\end{rem}

\begin{example}
As an example, let us discuss the case of a Fano manifold $F$ of Picard number $2$.

In this case there are only two possibilities: either $\Mon(F)$ is trivial or it has order $2$. 
In fact, as $\Nef(F)$ is invariant, if the action of $\Mon(F)$ is not the trivial one, then it 
permutes the two primitive vectors $v_1, v_2$ generating the extremal rays of the cone.
 
When this case occurs, the canonical divisor lies on the line spanned by $v_1+v_2$ and 
Theorem \ref{Monodromy_And_MMP} shows that the extremal rays of $\Nef(F)$ correspond to the same type of contraction in the $K_F$-MMP and the images of $F$ under the two contractions are deformation equivalent.

Another easy example is for del Pezzo surfaces (cf. Section \ref{sec_3fold}): the generic del Pezzo surface of 
degree 3 has no automorphisms ($\HMon$ is then trivial in this case), but $\Mon$ is certainly not trivial, as we will 
show in Section \ref{sec_crit2}.
\end{example}

The group $\Mon(F)$ is invariant under flat deformations which preserve the nef cone, whereas the group $\HMon(F)$ can jump when we deform $F$. Let us state the definitive form of our result for further references.

\begin{definition}[Isotrivial fibration] 
A morphism 
\[f\colon X\rightarrow Y\]
is isotrivial if there exists an open dense subset $U$ of $Y$ such that all the geometric fibres over $Y$ are isomorphic to a fixed variety.          
\end{definition}
Equivalently, every point $t$ in $U$ has an Euclidean neighbourhood over which $f$ is holomorphically trivial; moreover, there exists a finite \'etale cover $U' \to U$ such that $X_{U'} \cong X_U \times_U U'$ is a trivial family (cf. \cite[Proposition 2.6.10]{Ser}).     

When we are dealing with isotrivial fibrations, the identification between fibres over two distinct points is given by a holomorphic automorphism of $F_t$. This does not give us a homomorphism from $\pi_1(U,t)$ to $\Aut(F_t)$, but it allows us to assume that the action of $[\g]$ is induced by a (non-unique) element of $\Aut(F_t)$. We remark that isotriviality is a special condition. The only case when it is granted for free is when $F_t$ is rigid. To summarise, in these two sections we proved the following result.

\begin{theorem}\label{Deligne}
Let 
\[f\colon X\rightarrow Y\]
be a dominant morphism of projective normal varieties, where $X$ is $\qq$-factorial with rational singularities and on an open dense subset of $Y$ the fibres of $f$ are terminal and $\qq$-factorial. Assume that $-K_X$ is $f$-ample. Then there exists a maximal open dense subset $U=U'_f$ of $Y$ such that
\begin{itemize}
 \item the morphism $f\colon X_U \to U$ is a flat family of $\qq$-factorial Fano varieties with terminal singularities;
 \item for every $t$ in $U$, the monodromy action of $\pi_1(U,t)$ on $N^1(F_ t)_\qq$ factors through the finite group $\Mon(F_t)$ defined above. If the fibration is isotrivial, the monodromy factors through $\HMon(F_t)$.
\end{itemize}

Moreover, the restriction map
\[\rho \colon N^1(X)_{\Q}\rightarrow N^1(F_t)_{\Q}^{\pi_1(U,t)}\]
is surjective for every $t$ in $U$.
\end{theorem}

Since in the following we will focus on a specific class of Fano fibrations, let us recall a definition.

\begin{definition}\label{mfs.def}
Let $f\colon X \rightarrow Y$ be a dominant projective morphism of normal varieties. Then $f$ is called a {\it Mori fibre space} (or simply {\it MFS}) if the following conditions are satisfied:
\begin{enumerate}
\item $f$ has connected fibres, with $\dim Y < \dim X$;
\item $X$ is $\Q$-factorial with at most Kawamata log terminal singularities;
\item the relative Picard number of $f$ is one and $-K_X$ is $f$-ample.
\end{enumerate}
\end{definition}

We can finally introduce the key notion for our purposes.

\begin{definition}[Fibre-like]\label{fibre_like}
A normal Fano variety $F$ with terminal $\mathbb{Q}$-factorial singularities is said to be {\it fibre-like} if it can be realised as a fibre of a Mori Fibre Space $f\colon X \rightarrow Y$ over $U'_f$, where $U'_f$ is as in Theorem \ref{Deligne}.
\end{definition}

In all the examples we produce, the total space $X$ will have quotient singularities, which are well known to be klt.
\end{subsection}
\end{section}

\begin{section}{Criteria for Fibre-likeness}\label{sec_crit2}
\begin{subsection}{General Criteria}

In this section we present two criteria, one sufficient and one necessary, which detect the fibre-likeness in a rather general setting. The necessary criterion is based on Theorem \ref{Deligne}. When the Fano variety is rigid, we obtain a characterisation. Here by rigid we simply mean that $H^1(F, T_F)=0$.

\begin{theorem}[Sufficient Criterion]\label{sufficientCriterion}
A Fano variety $F$ with terminal $\mathbb{Q}$-factorial singularities and such that
\[
N^1(F)_\qq^{\Aut(F)} =\qq [K_{F}]
\]
is fibre-like.

Moreover, there exists a Mori fibre space $f\colon X \to Y$ such that the base $Y$ is a curve and the fibration is isotrivial.
\end{theorem}

\begin{remark}\label{suff.crit.rmk}
Before giving the proof, let us remark that $K_F$ is always fixed by $\Aut(F)$. 
In particular, if we fix $m \in \mathbb{N}$ s.t. $-mK_F$ is very ample, then the action of $\Aut(F)$ lifts faithfully to a linear action on $|-mK_F|$.
In other words, the hypothesis of the theorem is requesting that the subspace of $N^1(F)_{\Q}$ fixed by $\Aut(F)$ is minimal.
\end{remark}

\begin{proof}[Proof of Theorem \ref{sufficientCriterion}]
We know that $\HMon(F)$ is finite since the nef cone $\Nef(F)$ of $F$ is rational polyhedral and $\HMon(F)$ permutes its faces.

Pick a set of generators $[f_1],\dots , [f_g]$ of $\HMon(F)$. 
Call $G$ the sub-group of $\Aut(F)$ generated by $f_1,\dots , f_g$. 
Take a genus $g$ curve $C$ and denote by $a_i$ and $b_i$ the generators of its fundamental group. 
There is a unique relation between the $a_i$ and $b_i$ and it is the one on the product of commutators:
\[
a_1b_1a_1^{-1}b_1^{-1}a_2b_2a_2^{-1}b_2^{-1}\cdots a_gb_ga_g^{-1}b_g^{-1}=1.
\]

We define a surjective morphism
\begin{displaymath}
\begin{array}{cccc} 
\rho \colon &\pi_1(C,t) &\rightarrow &G.\\ 
& a_i& \mapsto &f_i \\
& b_i&\mapsto & f_i^{-1}
\end{array}
\end{displaymath}

Let $\hat{C}$ be the universal cover of $C$. We define 
\[
X:=F\times \hat{C} \Big/\pi_1(C,t),
\]
where $\pi_1(C,t)$ acts on $F$ via $\rho$. The action of $\pi_1(C, t)$ is free and properly discontinuous; 
hence $X$ is an analytic space with terminal singularities, i.e., the same type of singularities of $F$. 
The natural projection
\[
f\colon X\rightarrow C
\]
is an isotrivial fibration with fibre isomorphic to $F$.

Let us show that $X$ is projective. Indicating by $\phi_{|-mK_F|}\colon F \to \mathbb{P}^N$, $N=\dim |-mK_F|$ the embedding induced by the linear system 
of $|-mK_F|$, we know from Remark \ref{suff.crit.rmk} that the action of $\pi_1(C, t)$ extends also to $\mathbb{P}^N$ 
and the map $\phi_{|-mK_F|}$ is equivariant for the action.

Hence we have the following commutative diagram
\begin{displaymath}
\xymatrix{
F\times \hat{C} \ar[r]^{\phi_{|-mK_F|} \times id_{\hat{C}}} \ar[d] & \mathbb{P}^N\times \hat{C} \ar[d]\\
X=F\times \hat{C} \Big/\pi_1(C,t) \ar[r]^\psi \ar[d]^f & \mathbb{P}^N \times \hat{C} \Big/\pi_1(C,t) =Z \ar[d]^g\\
C \ar[r]^{id} & C }
\end{displaymath} 

As above, the singularities of $Z$ are the same as $\mathbb{P}^N$ and so $Z$ is smooth.
Moreover, as the action of $\Aut(F)$ is contravariant for $\phi_{|-mK_F| \times id_{\hat{C}}}$, $Z$ maps to $C$ and every fibre is isomorphic to $\mathbb{P}^N$. In particular, the anticanonical sheaf $\mathcal{O}_Z(-K_Z)$
is relatively ample over $C$. Since $C$ is itself projective, it follows that $Z$ is projective. 

The variety $X$ is $\mathbb{Q}$-factorial. This is simply a consequence of \cite[Cor. 12.1.9]{kollarmori}, 
as in view of the hypotheses of the theorem, $F$ is Cohen-Macaulay, in particular $S_3$, \cite[Thm. 5.22]{SingularityMMP} and the codimension of the singular locus of $F$ is at least $3$.

To finish the proof, we need to show that $\rho(X/C)=1$. We fix a point $t$ on $C$ and consider the fibre $F_t$ of $f$ over $t$. That is isomorphic to $F$ via the map $q$ defined in the above diagram. We will denote the inclusion of $F_t$ in $X$ by $\iota\colon F \hookrightarrow X$. As $\rho(C)=1$ it suffices to show that $\rho(X)=2$. We consider the sequence
\begin{eqnarray}\label{exact.sequence.4.4}
0\rightarrow N^1(C)_\qq \xrightarrow {f^*} N^1(X)_\qq\xrightarrow {\iota^*}N^1(F_t)_\qq^G\rightarrow 0.
\end{eqnarray}
If this sequence is exact, then $\rho(X)=2$.

The injectivity of $f^*$ follows by the connectedness of the fibres and the projection formula. The vector space $N^1(F_t)_\qq^G$ is generated by $-K_{F_t}$. By adjunction $\iota^*K_X=K_{F_t}$, so $\iota^*$ is surjective. We have $\im f^*\subseteq \ker\iota^*$. We need to show the opposite inclusion, but this follows by the same reasoning as in the proof of Lemma \ref{sur} from the seesaw principle.
\end{proof}
\begin{remark}
The same proof works verbatim in the case where $X$ is a Fano variety with rational $\mathbb{Q}$-factorial singularities of dimension $\geq 3$ and $F$ is smooth in codimension $2$, i.e. the singular locus has codimension at least $3$. 
In fact the assumption on the terminality of $X$ has been used only to allow us to use \cite[Cor. 12.1.9]{kollarmori}. But all the hypotheses in the corollary are verified with the weaker assumptions just explained.
\end{remark}

We state now a necessary criterion for $F$ to be fibre-like.

\begin{theorem}[Necessary Criterion]\label{necessaryCriterion}
A Fano variety with $\mathbb{Q}$-factorial terminal singularities for which
$$\dim N^1(F)_\qq^{\Mon(F)} > 1$$
is not fibre-like.
\end{theorem}
\begin{proof}
We argue by contradiction. Let 
\[f\colon X \rightarrow Y\]
be a Mori fibre space. 
By the definition of fibre-likeness, there exists an open dense subset $U=U'_f$ of $Y$ such that the map
\[\rho\colon N^1(X)_{\Q}\rightarrow N^1(F_t)_{\Q}^{\pi_1(U,t)}\]
for a given fibre $F_t$ isomorphic to $F$. 
Let us show that $N^1(F_t)_{\Q}^{\pi_1(U,t)}$ is one dimensional. We first prove that sequence
 \[0\rightarrow N^1(Y)_\qq \xrightarrow {f^*} N^1(X)_\qq\xrightarrow {\rho}N^1(F_t)_\qq^{\pi_1(U,t)}\rightarrow 0.\]
is exact. Since $F$ is connected, the map $f^*$ is injective on $N^1(Y)_\qq$. The inclusion $\im f^*\subseteq \ker\rho$ holds because the composition $F_t\to X \to Y$ factors through a point. 
The map $\rho$ is not the zero map. Now, we use the fact that we are dealing with a Mori fibre space.  The relative Picard number is one, so:
\[\dim N^1(X)_{\Q}=\dim N^1(Y)_{\Q}+1.\]
Thus, the sequence is exact and 
$$ \dim N^1(F_t)_\qq^{\pi_1(U,t)} =1 $$
 
By Theorem \ref{Deligne}, the monodromy action factors through $\Mon(F)$, so
\[N^1(F)_{\Q}^{\Mon(F)}=\Q K_F.\]
This contradicts our hypothesis.
\end{proof}
In the next section we will introduce a more handy version of this criterion.  Let us finish this section by considering the rigid case.

\begin{theorem}[Characterisation - Rigid case]\label{rigidchar}
A rigid Fano variety $F$ with $\mathbb{Q}$-factorial terminal singularities is fibre-like if and only if
\[N^1(F)_{\Q}^{\Aut(F)}=\Q K_F\]
In this case, $F$ is a fibre of an isotrivial Mori fibre space over a curve.
\end{theorem}
\begin{proof}
The ``if'' part is Theorem \ref{sufficientCriterion}. 
The ``only if'' part follows from Theorem \ref{necessaryCriterion}: just remark that if $F$ is rigid the monodromy action factors through 
$\HMon(F)$ (cf. Theorem \ref{Deligne}).
\end{proof}
If $F$ is not rigid this characterisation is false. A counterexample is the del Pezzo surface of degree 3 (see Section \ref{sec_3fold}).

\end{subsection}
\begin{subsection}{Applications of the Necessary Criterion}

The group $\Mon(F)$, defined in \ref{MonGroup}, is in general difficult to describe. Roughly speaking, it can be thought of as the group of symmetries of the nef cone preserving some other features coming from the birational geometry of the underlying variety. Taking this point of view, we can rephrase this criterion in terms of the birational geometry of $F$. The idea is that, if the faces of the nef cone are different from the view point of birational geometry, then $N^1(F)_\qq^{\Mon(F)}$ must be big. Let us give an easy example. Assume that $F$ has Picard number 2. The nef cone has two faces, $\mathcal{G}_1$ and $\mathcal{G}_2$. Each face gives a contraction
\[\pi_i \colon F \rightarrow G_i.\]
\begin{corollary}
Keep notations as above. If
\[\dim G_1 \neq \dim G_2\]
then $F$ can not be fibre-like.
\end{corollary}
\begin{proof}
The group $\Mon(F)$ can not exchange $\mathcal{G}_1$ and $\mathcal{G}_2$, so it is trivial.
\end{proof}
Case by case, one can cook up more refined versions of this corollary. Let us give more examples.

\begin{corollary}\label{unique.blowdown.prop}
Let $F$ be a Fano variety obtained as the blowup of another Fano variety $G$ and assume there are no other facets of $\Nef(F)$ whose associated contraction is divisorial with image a variety deformation equivalent to $G$. Then, $F$ cannot be a fibre-like Fano.
\end{corollary}

\begin{proof}[Proof of Corollary \ref{unique.blowdown.prop}]
The face of $\Nef(F)$ corresponding to the pullback of $\Nef(G)$ is invariant by 
$\Mon(F)$. In fact, the type of an extremal contraction and the deformation type of its image are 
preserved under the action of $\Mon(F)$. Hence, the uniqueness implies 
that the map must be preserved by such action. \newline
It is enough to show that on such face there is a fixed point and, consequently, a fixed one-dimensional subspace. As, in order to be a fibre-like Fano, the only subspace preserved by $\Mon(F)$ could be the span of $K_F$ and this does not lay on the pullback of $\Nef(G)$, the required contradiction 
is immediate. As we explained above, if $\Nef(G)$ is stable by $\Mon(F)$, then the class of the 
exceptional divisor is fixed as well.
\end{proof}

The above criterion can be generalised quite easily. Let us explain how, by means of some examples.

\begin{example}\label{cone.p1xp1.example}
Let $F$ be a Fano manifold that possesses a unique Mori fibre contraction to a variety $G$, with $\dim G = k > 0$.
Then the face of the nef cone of $F$ corresponding to the nef cone of $G$ is stable under $\Mon(F)$. 
In particular, the primitive generators of the extremal rays (in the lattice $N^1(F)\subset N^1_\rr(F)$) of such a face 
are going to be permuted by $\Mon(F)$. In particular their sum will be $\Mon(F)$-invariant. Hence, $F$
cannot be fibre-like. This is the case, for example, for the projectivisation $F$ of the vector bundle associated to the 
sheaf $\oo_{\pp^1 \times \pp^1} \oplus \oo_{\pp^1 \times \pp^1}(1, 1)$.
Recall that $F$ is isomorphic to the blow-up of the cone over a smooth quadric in $\pp^3$ with center the vertex.
$\rho(F)=3$ and the facets of $\Nef(F)$ are given by the Mori fibre contraction $F \to \pp^1 \times \pp^1$ 
and the two small contractions $F \to F_i, \; i=1,2$, given by contracting the two rulings of the exceptional copy of
$\pp^1 \times \pp^1$.
\end{example}

The above analysis can be formalised into the following statement.
\begin{corollary}\label{def.equiv.cor}
Let $F$ be a Fano variety and assume that the nef cone of $F$ contains a facet $\mathcal{G}$ corresponding to a 
certain variety $G$. 
Assume that for any other facet $\mathcal{H}$ of the nef cone, the corresponding variety $H$ is not 
deformation equivalent to $G$. Then, $F$ cannot be a fibre-like Fano.
\end{corollary}

So far we have dealt with the case of a facet globally fixed by $\Mon(F)$. 
What happens to facets that are translated around the nef cone?

Let $F$ be a Fano variety and $\mathcal{F}$ be a facet of $\Nef(F)$ and let $L$ 
be the sum of the primitive generators of the extremal rays spanning $\mathcal{F}$.
Let $\mathcal{F}_1, \dots, \mathcal{F}_k$ be the facets corresponding to translates of $\mathcal{F}$ under $\Mon(F)$. Again, for each of the facets $\mathcal{F}_1, \dots, \mathcal{F}_k$, let $L_1, \dots, L_k$ 
be the sum of the primitive generators of the extremal rays spanning the facet. 
The $L_i$ constitute the orbit of $L$ under the action of $\Mon(F)$. Hence, $L_1+\dots +L_k$ is $\Mon(F)$-invariant. 
In order for $F$ to be fibre-like, it has to be a negative multiple of $K_F$, in particular it has to be ample.

When $\mathcal{F}$ corresponds to a divisorial contraction, then the same reasoning applies to show that the sum $E+E_1+\dots +E_k$ of the exceptional divisors relative to the different facets must be a multiple of $-K_F$; otherwise $F$ will not be fibre-like.

\begin{example}\label{mult.facets.example}
Let $Q$ a smooth quadric $Q \subset \pp^3$. Let $p\colon R \to Q$ be the projective space bundle 
$\pp(\oo_Q \oplus \oo_Q(1, 1))$.\newline
The map $p$ has two sections $E_0, \; E_{(1, 1)}$ corresponding to the two projections of $\oo_Q \oplus \oo_Q(1, 1)$ on its factors.\newline
Let $F$ be the Fano variety obtained as the blow-up of $R$ along an 
elliptic curve $C$ contained in $E_0$. We will denote by $\pi\colon F \to Q$ the given map.\newline
The generic fibre of $\pi$ is $\pp^1$, but over 
$p(C) \subset Q$ the fibres are chains of two copies of $\pp^1$ intersecting at a point.
The variety $F$ has exactly two different divisorial contractions $\psi_i :F \to \pp_Q(\oo_Q \oplus \oo_Q(1, 1)), \; i=1,2$ 
given by contracting the two components of the fibres of $\pi$ over $p(C)$, respectively. 
Hence, for $F$ to be fibre-like, the sum of the exceptional divisors for the $\psi_i$ must be ample.
But this is not possible as such sum has intersection $0$ with the generic fibre of $\pi$.
\end{example}
\end{subsection}
\end{section}
\begin{section}{Surfaces, threefolds and other higher dimensional examples}\label{sec_3fold}
\begin{subsection}{Del Pezzo surfaces}
The classification of fibre-like surfaces was carried out in \cite[Theorem 3.5, Addendum to item 3.5.2]{Mori_surf} when the total space $X$ has dimension three.  We generalise this result to higher dimensional total spaces. Let us denote by $S_d$ the blow up of $\pp^2$ at $9-d$ general points.
\begin{theorem}\label{mori_surf}
A del Pezzo surface $S$ is fibre-like if and only if it is isomorphic to $\pp^2$, $\pp^1\times \pp^1$ or $S_d$, with $d\le6$.
\end{theorem}
\begin{proof}
To show that $S_7$ and $S_8$ are not fibre-like we can apply either Theorem \ref{necessaryCriterion} or Theorem \ref{rigidchar}. To show that $\pp^1\times \pp^1$ and $S_d$, with $d\leq 6$ and $d\neq 3$, are fibre-like we can apply Theorem \ref{sufficientCriterion} and the classical analysis of the automorphism group of del Pezzo surfaces (cf. \cite{koitabashi} and \cite{dolgachev}). Now, let $F$ be $S_3$, a smooth cubic in $\pp^3$. A generic $F$ does not have automorphisms (cf. \cite{segre}) so we can not apply the sufficient criterion in Theorem \ref{sufficientCriterion}; however, we can argue as follows. Let $X$ be the incidence variety in $\pp^3\times \pp H^0(\pp^3,\oo(3))^{\vee}$. This is a smooth ample divisor. Hence by Lefschetz hyperplane theorem it has Picard number 2. The projection
$$ X\to \pp H^0(\pp^3,\oo(3))^{\vee} $$
is a Mori fibre space which contains all cubic surfaces as smooth fibres, so $S_3$ is fibre-like. We remark that we can handle in a similar way also $\pp^1\times \pp^1$ and $S_d$ with $d\leq4$.
\end{proof}
The cubic surface is an example of a fibre-like variety where the sufficient criterion in Theorem \ref{sufficientCriterion} does not apply. A consequence of our case-by-case proof is the following.
\begin{corollary}\label{Cor_Surf}
When $F$ is a surface, the necessary criterion \ref{necessaryCriterion} is actually a characterisation of fibre-likeness. Moreover, fibre-likeness is preserved by smooth deformations.
\end{corollary}

By comparing our result with the classification of $K$-stable smooth del Pezzo surfaces (\cite[Theorem 1.4]{Tosatti}), we also obtain the following corollary.

\begin{corollary}\label{delpezzo_stab}
A smooth del Pezzo surface is $K$-stable if and only if it is fibre-like.
\end{corollary}
\end{subsection}
\begin{subsection}{A general procedure to construct Mori fibre spaces}

The abstract notation will be heavy, so we start with an example. Let $F$ be a smooth divisor of bidegree $(2,2)$ in $\pp^2\times \pp^2$. Let $\sigma$ be the involution of $\pp^2\times \pp^2$ and $\pp^N:=\pp H^0(\pp^2\times \pp^2,\mathcal{O}(2,2))^{\vee}$. Consider the incidence variety $I$ in $\pp^N\times \pp^2\times \pp^2$; it is a smooth divisor of degree $(1,2,2)$. We can apply Lefschetz hyperplane theorem to show that the Picard number of $I$ is $3$. We have a fibration
$$\pi \colon I\to \pp^N,$$
whose relative Picard number is $2$. The involution $\sigma$ acts on this fibration. Let $X:=I/\sigma$ and $Y=\pp^N/\sigma$. In this way, we obtain a fibration
$$f\colon X \to Y,$$
with relative Picard number $1$. The singularities are finite quotient singularities, they are klt and $\qq$-factorial by \cite[Proposition 5.15 and Corollary 5.21]{SingularityMMP}. We conclude that $F$ is a Mori fibre space. By moving $F$ by a generic element of $\PGL(3)\times \PGL(3)$, we can always assume that it is not preserved by $\sigma$; hence the action of $\sigma$ is free on a neighbourhood of $F$ in $Z$.  This means that $F$ is a smooth fibre of $f$. The previous argument shows that every smooth divisor of bidegree $(2,2)$ in $\pp^2\times \pp^2$ is fibre-like.

We now generalise this construction. Let $F$ be a smooth Fano variety and let $Z$ be a smooth projective variety in which $F$ is immersed. Let $L_1, \dots, L_k$ be effective prime divisors on $Z$ such that the associated line bundles $\mathcal{O}_Z(L_i)$ are basepoint-free. We will indicate by $|L_i|$ the linear systems of the divisors $L_I$. \newline
Suppose that $F$ is a complete intersection $Z=L_1 \cap \dots \cap L_k$. Let $I$ be the incidence variety in $Z\times |L_1| \times \cdots \times |L_k|$ defined as
\[
Z:=\{ (z, D_1, \dots, D_k) \in Z\times |L_1| \times \cdots \times |L_k| \; | \; z \in D_1 \cap \dots \cap D_k\};
\]
 the variety $I$ is smooth since the $L_i$ are basepoint-free.  
\begin{lemma}\label{sur}
The restriction morphism
$$ \rho \colon \Pic(Z\times |L_1| \times \cdots \times |L_k|) \to \Pic(I) $$
is surjective.
\end{lemma}
\begin{proof}
Consider the projection
$$ \pi \colon I \to Z.$$
The fibres are divisors of multi-degree $(1,\dots,1)$ in $ |L_1| \times \cdots \times |L_k| $. In particular, they are equidimensional, hence $\pi$ is flat. Fix a point $t$ in $Z$ and let $H$ be the fibre over it. Below, we will show that the sequence
\begin{displaymath} \label{ex.seq.sur}
\Pic(Z) \to \Pic(I) \to \Pic(H)
\end{displaymath}

is exact. 
Since the image of $  \Pic(Z\times |L_1| \times \cdots \times |L_k|)$ contains the image of $\Pic(Z)$ and surjects to $\Pic(H)$ we obtain the statement.

To complete the proof of the exactness of the sequence in \ref{ex.seq.sur} we argue as in the last part of the proof of Theorem \ref{sufficientCriterion}. 
Let $L$ be a line bundle on $I$ whose restriction to $H$ is trivial. Since $R^2\pi_*\qq$ is locally constant on $Z$, we have that $c_1(L)$ is trivial on every fibre. The fibres are Fano, so $L$ itself is trivial on every fibre. The map $\pi$ is flat, so, by the seesaw principle, cf. \cite[Corollary 6, p. 54]{Mumford} or \cite[Proposition 12.1.4]{kollarmori}, a line bundle which is trivial on each fibre is the pull-back of a line from the base.

\end{proof}
Suppose that there is a finite subgroup $G$ of $\Aut(Z)$ which is fixed-point-free in codimension one and whose action can be lifted to $I$; fix such a lifting.  Assume that $G$ does not preserve $F$.
\begin{theorem}\label{GeneralConstruction}
Keep notation as in Lemma \ref{sur}. If
$$  \dim N^1(Z)_\qq^G=1,$$
then $F$ is fibre-like.
\end{theorem}
\begin{proof}
We construct explicitly a Mori fibre space which will have $F$ as a smooth fibre. Let $X:=I/G$ and $Y= (|L_1| \times \cdots \times |L_k|)/G$. We claim that the projection 
$$ f\colon X \to Y  $$
has relative Picard number one. 
Since the $|L_i|$ are projective spaces, we have  
$$\Pic(Z\times |L_1| \times \cdots \times |L_k|) =\Pic(Z)\times \Pic(|L_1|)\times \cdots \times \Pic(|L_k|).$$
Lemma \ref{sur} and the hypothesis $\dim N^1(Z)_\qq^G=1$ imply that
\begin{eqnarray}\nonumber
\dim N^1(X)_\qq=\dim N^1(I)_\qq^G &\leq \sim N^1(X \times |L_1| \times \cdots \times |L_k|)_\qq^G=\\ \nonumber
= \dim N^1( |L_1| \times \cdots \times |L_k|)_\qq^G+1&=\dim N^1(Y)_\qq+1.
\end{eqnarray}
The variety $F$ is a smooth fibre of $f$ because it is not fixed by $G$. Since the action of $G$ on $Z$ is fixed-point-free in codimension one, the singularities of $X$ and $Y$ are klt and $\qq$-factorial by \cite[Proposition 5.15 and Corollary 5.21]{SingularityMMP}.
\end{proof}
We remark that it should not be easy to check if the singularities are terminal, as explained in the remark after \cite[Corollary 5.21]{SingularityMMP}. Let us apply our result. Denote by $(\pp^n)^r$ the cartesian product of $r$ copies of $\pp^n$.
\begin{corollary}\label{esempi_in_dimensione_alta}
Take positive integers $r,k,d,$ and $n\geq 2$ such that $kd<n+1$. Let $F$ be a smooth complete intersection of $k$ divisors of degree $(d,\dots ,d)$ in $(\pp^n)^r$. Then $F$ is fibre-like.
\end{corollary}
\begin{proof}
The condition $kd<n+1$ ensures that $F$ is Fano.
Let $G$ be the symmetric group on $r$ elements. It acts on $(\pp^n)^r$ permuting the factors. By acting by a general automorphism of $(\pp^n)^r$ we can arrange that $G$ does not fix $F$. We can now apply Theorem \ref{GeneralConstruction}. 
\end{proof}
The Mori fibre space will in general depend on a choice of the lifting of $G$ to a subgroup of $\Aut(I)$.  
For instance, if $G$ acts trivially on the linear systems, the base $Y$ will be a product of projective spaces. If the lifting is nontrivial, the base will be a singular variety with smaller Picard number. Moreover, we could have chosen a smaller $G$. It is enough that the action of $G$ is transitive on the copies of $\pp^n$ and fixed-point-free in codimension one on $(\pp^n)^r$.
\begin{corollary}\label{esempi_in_dimensione_alta_due}
Fix two integers $r$ and $n\geq 2$. Denote by $L_i$ be the line bundle $\oo(1,\dots, 0, \dots ,1)$ on $(\pp^n)^r$, where the $0$ appears only at the $i$-th position. A smooth complete intersection $F$ of multi-degree $(L_1,\dots, L_r)$ in $(\pp^n)^r$ is fibre-like.
\end{corollary}

Clearly, there are many variants of these corollaries.In the next section we will give a few more specific examples.

\end{subsection}
\begin{subsection}{Fano Threefolds}

The results described in the previous sections show that there are quite a few restrictions on the geometry of a Fano variety
$F$ in order for it to be fibre-like. We are interested in understanding
how strong these restrictions are. As vague as this question may appear, drawing on the classification of smooth Fano
threefolds due to Mori and Mukai (cf. \cite{mori.mukai.class} and \cite{mori.mukai.class.2}),
we are able to show that in this context most threefolds do not satisfy these restrictive conditions.

We will refer to \cite[Tables 2, 3, 4, 5]{mori.mukai.class} where a full description of the deformation types of
Fano threefolds is given. 

\begin{theorem}\label{3folds.MFS.thm}
Let $F$ be a smooth Fano threefold with Picard number greater than $1$. Then $F$ is fibre-like if and only if its deformation type appears in Table \ref{3folds.MFS.table}.
\end{theorem}
\begin{table}\label{3folds.MFS.table}
\begin{center}

\begin{tabular}{*{4}{c} | p{8.5cm}}
No & \cite{mori.mukai.class} &$\rho(F)$ & $-K_F^3$ & Deformation type of $F$\\
\hline
1a & (6a) & 2 & 12 & $F$ is a divisor of bidegree $(2, 2)$ in $\pp^2 \times \pp^2$.\\
1b & (6b) & 2 & 12 & $F$ is a $2:1$ cover of a smooth divisor $W$ of bidegree $(1, 1)$ in $\pp^2 \times \pp^2$ branched along a member of $|-K_W|$.\\
2 & (12) &2 & 20 & $F$ is the blow-up of $\pp^3$ with center a curve of degree 6 and genus 3 which is an 
intersection of cubics. Alternatively, $F$ is the intersection of three divisors of bidegree $(1, 1)$ in 
$\pp^3 \times \pp^3$.\\
3 & (28) &2 & 28 & $F$ is the blow-up of $Q \subset \pp^4$ with center a twisted quartic, a smooth rational curve
of degree 4 which spans $\pp^4$. \\
4 & (32) &2 & 48 & $F$ is a divisor of bidegree $(1, 1)$ in $\pp^2 \times \pp^2$.\\
5 & (1) &3 & 12 & $F$ is a double cover of $\pp^1 \times \pp^1 \times \pp^1$ whose branch locus is a divisor of
tridegree $(2, 2, 2)$.\\
6 &(13) &3 & 30 & $F$ is the blowup of a smooth divisor of bidegree $(1, 1)$ in $\pp^2 \times \pp^2$ with center a 
curve $C$ of bidegree $(2, 2)$  on it, such that $C \hookrightarrow W \hookrightarrow \pp^2 \times \pp^2 
\to \pp^2$ is an embedding for both both projections $\pp^2 \times \pp^2 \to \pp^2$.\\
7 & (27) &3 & 48 & $F= \pp^1 \times \pp^1 \times \pp^1$.\\
8 & (1) &4 & 24 & $F$ is a smooth divisor of multi degree $(1, 1, 1, 1)$ in $\pp^1 \times \pp^1 \times \pp^1 \times \pp^1$.\\
\end{tabular}
\caption{Deformation types of Fano varieties in Theorem \ref{3folds.MFS.thm} }

\end{center}

\end{table}
\begin{remark}
In the second column of Table \ref{3folds.MFS.table} we use the numbering adopted in
\cite{mori.mukai.class}. Exactly the same numbering will be used throughout our proof. 
We remark that entry $1a$ and $1b$ have the 
same deformation type. Alternative descriptions of these manifolds, which we will use, can be found in 
\cite{prok.Gfano.2}.
\end{remark}

\begin{rem}\label{last.column.label}
The last column in each table presented in \cite{mori.mukai.class} enumerates
all the possible ways a Fano threefold can be obtained from another Fano threefold by blowing up a curve. 
Alternatively, in the language of this section, they describe all the facets of the nef cone corresponding to a divisorial 
contraction in which the image of the exceptional divisor is a curve. We remark that the contractions are listed without multiplicity; this means that there could be more than one face giving the same contraction.
\end{rem}

\begin{proof}
For the reader's convenience, we will divide our analysis based on the Picard number of the Fano threefolds that we examine.

{\bf Fano varieties of Picard number $2$}

The nef cone of a Fano variety $F$ of Picard number 2 is a rational polyhedral cone of the form 
$\rr^+D_1 + \rr^+ D_2$, for $D_1, D_2$ two nef, semiample (integral) Cartier divisors on $F$. 
In this representation, we always assume that the classes of the $D_i$ are primitive in the N\'eron-Severi group. 
\begin{rem}\label{rho2.rmk}
As the nef cone is $\Mon(F)$-invariant, $\dim \Nef(F)=2$ and the only invariant subspace for 
the action of $\Mon(F)$ is the span of the anticanonical class, it 
follows that the sum of the primitive generators of $\Nef(F)$ must be a multiple of the canonical class, i.e. there exists $\lambda <0$ such that
\[
\lambda K_F \sim D_1+D_2.
\]
This is another useful condition: 
e.g., a Fano variety $F$ isomorphic to a smooth divisor of type $(1, 2)$ 
contained in $\mathbb{P}^2 \times \mathbb{P}^2$ cannot be fibre-like.
Let $i \colon F \hookrightarrow \mathbb{P}^2 \times \mathbb{P}^2$ be the inclusion of 
$F$ in $\mathbb{P}^2 \times \mathbb{P}^2$. We denote by $p_1, p_2$ the projections of 
$\mathbb{P}^2 \times \mathbb{P}^2$ onto the its two factors.
By Lefschetz hyperplane theorem, $\Nef(F)=i^\ast\Nef(\mathbb{P}^2 \times \mathbb{P}^2)$. 
Then $\Nef(F)= \mathbb{R}^+ i^\ast p_1^\ast(\mathcal{O}_{\mathbb{P}^2}) + \mathbb{R}^+ i^\ast p_2^\ast(\mathcal{O}_{\mathbb{P}^2})$ and the two classes are the primitive generators of the cone.
The adjunction formula implies that 
\[
K_F = (K_{\mathbb{P}^2 \times \mathbb{P}^2} + F)|_F= 
(K_{\mathbb{P}^2 \times \mathbb{P}^2}+\mathcal{O}_{\mathbb{P}^2 \times \mathbb{P}^2}(1, 2))|_F.
\]
It is immediately clear that $[K_F]$ is not contained on the line spanned by the sum of the two primitive
generators of $\Nef(F)$.
\end{rem}
Using Corollary \ref{def.equiv.cor},
we can immediately exclude the families corresponding to the following entries of 
Table 2 of \cite{mori.mukai.class}:
\[
(1-5), (7-11), (13-20), (22), (23), (25-31), (33-36).
\]

The variety corresponding to entry number $(12)$, the intersection of three divisors of bidegree $(1, 1)$ in $\pp^3 \times \pp^3$, is fibre-like because of Theorem \ref{GeneralConstruction}. Using Remark \ref{rho2.rmk}, we can also exclude entry $(24)$. 

The variety corresponding to entry $(6a)$ is a divisor of degree $(2,2)$ in $\pp^2\times \pp^2$, 
while the variety from entry $(32)$ is a divisor of type $(1, 1)$ in 
$\pp^2 \times \pp^2$. They  are fibre-like because of Corollary \ref{esempi_in_dimensione_alta}. 
Entry $(6b)$ is a $2:1$ cover $F$ of a smooth divisor $W$ of bidegree $(1, 1)$ in $\pp^2 \times \pp^2$ branched along a member of $|-K_W|$. We can construct inside $|-K_W| \times|-K_W|$ the universal family $Z$ for $F$ (cf. \cite[Chapter I.17]{BPV}). The variety $Z$ is smooth and projective and it has Picard number $3$. We remark that $W$ has an involution $\sigma$; its action can be lifted to both $|-K_W|$ and $Z$. 
Letting $Y:=|-K_W|/\sigma$ and $X:=Z/\sigma$ we obtain a Mori fibre space which contains $F$ as a general fibre (cf. Theorem \ref{GeneralConstruction}).

Entry $(28)$ can be described as a smooth complete intersection of $L_1:=f^*H$ and $L_2:=f^*2H-E$, where $f\colon Z \rightarrow \pp^5$ is the blow-up of the Veronese surface $V$, $E$ is the exceptional divisor and $H$ is an hyperplane in $\mathbb{P}^5$. We want to apply Theorem \ref{GeneralConstruction}. To this end we construct an order $2$ automorphism $C$ of $Z$ such that $C^*L_1=L_2$. The automorphism $C$ is a special Cremona transformation. The Veronese surface is the intersection of $6$ quadrics, so we have a Cremona transformation of $\pp^5$ whose indeterminacy locus is $V$. By blowing up $V$, we get a regular map from $Z$ to $\pp^5$ which contracts the secant variety of $V$, so this new map is again a blow-up. We conclude that $C$ lifts to a regular automorphism of $Z$. One checks that it acts nontrivially on the Picard group. A general reference for this kind of Cremona map is \cite{Cremona}.

{\bf Fano varieties of Picard number $3$}

Entry  $(1)$ is a double cover of $\pp^1 \times \pp^1 \times \pp^1$ whose branch locus is a divisor of
tridegree $(2, 2, 2)$. We will use the same notation as \cite[Section 54]{corti_complete}: this Fano variety can be realised as a member of the linear system $|2L+2M+2N|$ in the toric variety $Z$ with weight data
\begin{center}
\vspace{.2cm}
{
\begin{tabular}{*{7}{c} | *{1}{c} }\label{toric_data}
$x_0$ & $x_1$ & $y_0$  & $y_1$ & $z_0$ & $z_1$ & $w$ & \\
\hline
1 & 1 & 0 & 0 & 0 & 0 & 1 & $L$\\
0 & 0 & 1 & 1 & 0 & 0 & 1 & $M$\\
0 & 0 & 0 & 0 & 1 & 1 & 1 & $N$\\
\end{tabular}
}
\end{center}
Also for this variety Theorem \ref{GeneralConstruction} applies, since $Z$ carries a natural action of the symmetric group $\mathcal{S}_3$ which exchanges the divisors $L$, $M$ and $N$ and so lifts to the linear system $|2L+2M+2N|$.
This shows that entry $(1)$ is fibre-like.

Entry $(13)$ can be alternatively described as a smooth complete intersection of three divisors of multi-degree $(0,1,1)$, $(1,0,1)$ and $(1,1,0)$ in $\pp^2\times \pp^2\times\pp^2$. It is fibre-like because of Corollary \ref{esempi_in_dimensione_alta_due}.

Using Table 3 of \cite{mori.mukai.class} and
Corollary \ref{def.equiv.cor}, we can immediately exclude the families corresponding to the 
following entries of the table:
\[
(2-8), (11), (12), (14-16), (18), (20-24), (26), (28-31).
\]

\begin{rem}\label{inter.faces.rmk}
Let $F$ be a Fano variety of Picard number 3. Suppose that the nef cone contains two facets for which the images
of the corresponding contraction morphisms are deformation equivalent. Then these may be identified
by the action of $\Mon(F)$. In particular, the primitive generators of the two facets are exchanged and their sum is then
invariant. Hence it has to belong to the span of the canonical class, if $F$ is fibre-like.

When the two facets correspond to divisorial contractions the same holds true for the sum of the two exceptional 
divisors $E_i$, with $i=1,2$. In particular, $E_1+E_2$ has to be ample.
\end{rem}

Using the previous remark, the following entries can be shown not to be of fibre-like type:
\[
(3), (9), (10), (17), (19), (25).
\]

{\bf Fano varieties of Picard number $4$}

In Table 4 of \cite{mori.mukai.class} entry $(1)$, a divisor of multidegree $(1, 1, 1, 1)$ 
in $\pp^1 \times \pp^1 \times \pp^1 \times \pp^1$, is fibre-like because of Corollary \ref{esempi_in_dimensione_alta}. 
Using Corollary \ref{def.equiv.cor}, it is immediate to see that we can exclude the families corresponding to the 
following entries of the table:
\[
(3-6), (8-11), (13).
\]

Using the natural generalisation to Picard number $4$ of Remark \ref{inter.faces.rmk}, 
we can exclude the following entries, too:
\[
(2), (7), (12).
\]

{\bf Fano varieties of Picard number $5$}

In this case the only Fano threefolds are the following.
\begin{itemize}
\item Let $Y$ be the blow up of a quadric $Q \subset \pp^3$ along a smooth conic contained in it. 
The Fano variety $F$ is the blow-up of $Y$ with center three distinct exceptional lines of 
the blow-up $Y \to Q$; then the sum of the three 
exceptional divisors over the lines must be a (negative) multiple of $K_X$ and it is ample. That is 
clearly false true, as one can see by taking an exceptional line for the map $Y \to Q$ other than those already blown up.
\item $F$ is the blow-up of $Y=\pp(\oo_{\pp^1 \times \pp^1} (1, 0) \bigoplus \oo_{\pp^1 \times \pp^1}(0, 1))$ with center 
two exceptional lines $l_1, l_2$ of the blow-up $\phi\colon Y \to \pp^3$ such 
that $l_1$ and $l_2$ lie on the same irreducible component of the exceptional set of $\phi$; such $F$ is not
fibre-like by Proposition \ref{unique.blowdown.prop}.
\item Products 
\[
\mathbb{P}^1 \times S_d, \; d \leq 6,
\]
where $S_d$ is a del Pezzo of degree $d$. A quick analysis shows immediately that the projection onto the second factor
must be $\Mon(S_d)$-invariant as ${\rm Nef}(\mathbb{P}^1 \times S_d) = {\rm Nef}(\mathbb{P}^1) \times 
{\rm Nef}(S_d)$. 
\end{itemize} 
\end{proof}
As a consequence of our case-by-case proof we have the following corollary.
\begin{corollary}\label{Cor_3fold}
When $F$ is a threefold, the necessary criterion \ref{necessaryCriterion} is actually a characterization of fibre-likeness. Moreover, fibre-likeness is preserved by smooth deformations.
\end{corollary}

\begin{rem}[$K$-stability]\label{rui_st}
It is known that threefolds $(4)$ and $(7)$ are $K$-stable; in \cite{rui} varieties $(1.b)$ and $5$ are proved to be $K$-stable, being appropriate finite cover of $K$-stable varieties.
\end{rem}
\end{subsection}
\end{section}

\section{Smooth toric Fano varieties and K-stability}\label{sec_toric} 

In this section we prove that any smooth fibre-like toric Fano variety has barycentre in the origin (i.e. it is $K$-stable). Let us point out that there are smooth toric varieties that are $K$-stable but not fibre-like, such as $\pp^1\times \pp^2$. 

\subsection{Preliminaries on toric geometry: primitive collections}

We start recalling some notation and basic facts. For more details, see \cite{CLS11}, \cite{Bat91} and \cite{Cas03a}.

Let $N$ be a free abelian group of rank $n$ and set $N_\qq:=N\otimes_\zz \qq$. Denote by $M$ the dual of $N$.  Let $\Sigma \subset N_\qq$ be a fan of an $n$-dimensional smooth toric Fano variety $F$ and let $\De \subset N_\qq$ be the dual polytope associated to the anti-canonical polarisation. 

The polytope $\De$ is the polytope whose vertices are the integral generators of the $1$-dimensional cones contained in $\Sigma$. We denote the set of all vertices of $\De$ by $V(\De)$.

Let $N_1(F)$ be the group of 1-cycles on $F$ modulo numerical equivalence and set $N_1(F)_\qq=N_1(F) \otimes \qq$. Inside $N_1(F)_\qq$ we consider the Kleiman-Mori cone $\NE(F)$ generated  by the effective 1-cycles. There is the following basic exact sequence:
\begin{equation}\label{toricses1}
0 \rightarrow N_1(F) \rightarrow \zz^{V(\De)}\rightarrow N \rightarrow 0
\end{equation}
and dually
\begin{equation}\label{toricses2}
0 \rightarrow M \rightarrow \zz^{V(\De)}\rightarrow N^1(F) \rightarrow 0.
\end{equation}

In this subsection, we also need some notation and results about primitive collections.

\begin{definition}
A subset $\mathcal P \subset V(\Delta)$ is called a \emph{primitive collection} if the cone generated by 
$\mathcal  P$ is not in $\Sigma$ and for each $x \in \mathcal  P$ the cone generated by 
$\mathcal  P \setminus \{x\}$ is in $\Sigma$. 
\end{definition}
For a primitive collection $\mathcal  P=\{x_1, \ldots,x_k\}$ denote by $\sigma(\mathcal  P)$ the minimal cone in $\Sigma$ such that $x_1+\ldots +x_k \in \sigma(\mathcal  P)$ . Let $y_1,\ldots,y_h$ be generators of $\sigma(\mathcal  P)$.  By smoothness of $F$, there exist positive integers $b_i$ such that
$$
x_1+\ldots +x_k=b_1y_1+\ldots +b_hy_h;
$$
let $r(P)$ be this relation. 
\begin{definition}
The linear relation $r(\mathcal  P)$ is called \emph{the primitive relation of $\mathcal  P$} and the cone $\sigma(\mathcal  P)$ is called the focus of $\mathcal  P$. 
The integer $k$ is called the {\it length of $r\mathcal  (P)$} and the {\it degree of $\mathcal  P$} is defined as $\deg \mathcal  P= k - \sum b_i$.
\end{definition}

Using the exact sequence (\ref{toricses1}) we have the following identification between $N_1(F)$ and the group generated by relations among the vertices of $\De$:
$$
N_1(F) \cong \left\lbrace (b_x)_{x \in V(\De)} \in \Hom(\zz^{m}, \zz)\   \Bigg| \ \sum_{x \in V(\De)} b_x x=0\right\rbrace .
$$

\begin{rem}
By abuse of notation we denote by $r(\mathcal  P)$ also the cycle associated, via the previous isomorphism, to the relation
$$
x_1+\ldots +x_k - (b_1y_1 + \ldots + b_hy_h)=0.
$$ 
 Note that $\deg \mathcal  P = -(K_F \cdot r(\mathcal  P))$ and so, since we are considering Fano varieties, any primitive relation has strictly positive degree.
\end{rem}

We will need the following result (cf. \cite[Proposition 3.2]{Bat91}).

\begin{proposition}[Batyrev]\label{batyrev}
Let $F$ be a smooth toric Fano variety. Then there exists a primitive collection $\mathcal  P$ such that $\sigma(\mathcal  P)=0$.
\end{proposition}

\begin{rem}\label{rmk}
Let 
$$
a_1x_1+\cdots+a_kx_k=b_1y_1+\cdots+b_hy_h
$$
be a relation among vertices of $\De$ with all $\{a_i\}$ and $\{b_j\}$ positive integers. Assume that $\sum a_i \ge \sum b_j.$
Then, by Lemma 1.4 in \cite{Cas03a}, $\langle x_1\cdots, x_k\rangle \not\in \Sigma$.
\end{rem}

\subsection{Fibre-likeness implies K-stability}
It is known that the symmetry of the polytope $\Delta$ is related to the $K$-stability, which is known to be equivalent to the existence of a K\"ahler-Einstein metric (cf. \cite{WZ} and \cite{BB}) of the associated Fano variety. Mabuchi proved in \cite{mabuchi} the first result relating the $K$-stability with the triviality of the barycentre of $\Delta$. This result was generalised in the singular setting in \cite{Ber12}. 

\begin{theorem}[{\cite[Cor. 1.2]{Ber12}}]\label{berg}
Let $F$ be a Gorenstein toric Fano variety. Then $F$ is $K$-stable if and only if the barycentre of $\Delta$ is the origin.
\end{theorem}

The proof of the previous result is analytic and passes through the existence of K\"ahler-Einstein metrics. Applying the theorem above, we can see that there are Gorenstein terminal $\Q$-factorial  toric varieties which are fibre-like, but not $K$-stable, e.g., the weighted projective space $\mathbb P(1,1,1,1,2)$.

In this context our main result is the following.

\begin{theorem}\label{cor_kstability}
For every smooth toric fibre-like Fano variety $F$, the barycentre of the $\Delta$ is in the origin and, as a consequence,  $F$ is $K$-stable.
\end{theorem}

Before proving the theorem, we need to recall some convex geometry.

\begin{remark}
A basic fact is the following: the intersection of a convex polytope with an affine space is again a convex polytope.
\end{remark}

\begin{lemma}\label{lemma_convex_1}
Let $P$ be an $n$-dimensional convex polytope in an affine space $W\simeq\qq^n$ and let $H$ be a $k$-dimensional affine subspace intersecting the interior of $P$. Set $P':=P\cap H$ and consider a facet $\F'$ of $P'$. Then there exists a unique face $\F$ of $P$ of dimension al least $k-1$ such that
\begin{itemize}
\item $\F'=\F\cap H$;
\item $H$ intersects $\F$ in its relative interior.
\end{itemize}
\end{lemma}

\proof
The polytope $P$ is defined by a collection of inequalities $\{x \in W \ |\ a_i \cdot x \le b_i,\ i=1, \ldots t \}$, with $a_i,b_i \in \qq^d$, $t \ge n+1$ and $H$ is defined by a collection of equations $\{x \in W \ | \ c_j \cdot x = d_j, \ j=1,\ldots, d-k\}$, with $c_j,d_j \in \qq^d$. The facet $\F'$ is then defined, up to reordering the indices $i$ by $\{x \in W \ |\ a_i x = b_i,\ i=1, \ldots l, \  a_i x \le b_i,\ i=l+1, \ldots t,  \ c_j x = d_j, \ j=1,\ldots, d-k\}$, with $l\le n-k+1$. The set $\{x \in W \ |\ a_i x = b_i,\ i=1, \ldots l, \  a_i x \le b_i,\ i=l+1, \ldots t\}$ defines the unique face $\F$ of $P$ of dimension at least $k-1$ with the required properties.
\endproof

If we consider the action of a subgroup of automorphisms of the polytope on the vertices of $P$, we obtain the following lemma.

\begin{lemma}\label{lemma_convex_2}
Let $P$ be an $n$-dimensional polytope in an affine space $W\simeq\qq^n$ and let $G$ be a finite subgroup of $\GL(W,\qq)$. Assume that $P$ is invariant for the action of $G$ on $W$ and that $W^G$ is $k$-dimensional and intersects the interior of $P$. Then the action of $G$ on the vertices $V(P)$ has at least $k+1$ orbits.
\end{lemma}

\proof
We prove this lemma by induction on $k$. In the case $k=0$, there is nothing to prove. 
In the case $k=1$, the fixed locus $W^G$ is a line, which meets two distinct (possibly not maximal) faces $\F_1$ and $\F_2$ of $P$. We immediately obtain two invariant sets: 
$$V(\F_1)\setminus V(\F_2),\mbox{ \ \ and\ \ }V(\F_2)\setminus V(\F_1).$$ 
These give at least two orbits.

We now prove the inductive step. Since the intersection of a convex polytope with an affine space is again a convex polytope, we can consider the intersection polytope $P':=P\cap W^G$. Let $\F'$ be one of its $(k-1)$-dimensional facets. Using Lemma \ref{lemma_convex_1}, we can find a (unique) face $\F$ of $P$, cut in its interior by $W^G$ in a $k-1$ affine space such that $\F'= \F \cap W^G$. Let $H$ be the smallest affine subspace containing $\F$; it is preserved by the action of $G$. By induction, we obtain at least $k$ orbits of vertices contained in $\F$. The extra orbit is obtained by the vertices of $P$ not contained in $\F$.
\endproof

\proof[Proof of Theorem \ref{cor_kstability}]
Any smooth Fano toric variety is rigid (cf. \cite[Corollary 4.6]{dFH12}), so we can apply Theorem \ref{rigidchar}: $F$ is fibre-like if and only if $\dim N^1(F)_{\qq}^{\Aut(F)}=1$. 

After tensoring by $\qq$ the exact sequence (\ref{toricses2}), we obtain
 \begin{equation}\label{toricses_q}
 0\rightarrow M_\qq \rightarrow \mathbb{Q}^{V(\De)} \rightarrow N^1(F)_\qq \rightarrow 0,
 \end{equation}
where $M_\qq$ is the $n$-dimensional $\mathbb{Q}$-vector space containing the dual polytope of $\Delta$. There is a natural action of $\Aut(\Delta)$ on $M_\qq$ and $\qq^{V(\Delta)}$, and a natural homomorphism $\Aut(\Delta) \to \Aut(F)$, which make the sequence above equivariant for $\Aut(\Delta)$. Moreover by \cite[Corollary 4.7]{Cox95}) we have $N^1(F)_\qq^{\Aut(\Delta)}=N^1(F)_\qq^{\Aut(\F)}$. Let us denote by $t$ be the number of orbits of the action of $\Aut(\Delta)$ on $V(\De)$.

It is easy to see that if we take the $\Aut(\Delta)$-invariant part of the exact sequence in \ref{toricses_q}, we obtain again an exact sequence. Moreover, in this case it follows immediately that
\[
(\qq^{V(\Delta)})^{\Aut(\Delta)}=\qq^{V(\De)/\Aut(\De)}=\qq^t.
\]

Let now $G$ be $\Aut(\De)$ and $t$ be the number of orbits of the action of $G$ on $V(\De)$. Set $k:=\dim M^G$; the observation from the last paragraph and the sequence (\ref{toricses_q}) imply that $F$ is fibre-like if and only if $t-k=1$. Therefore, we want to prove that if $G$ has exactly $k+1$ orbits on $V(\De)$, then the barycentre of $\De$ is the origin.

Since we are working with $\mathbb{Q}$-vector spaces, $M$ and $N$ are isomorphic as $G$-modules; in particular $\dim N^G=\dim M^G=k$.

Let $\De'$ be the intersection polytope $N^G\cap \De$. For every facet $\F'$ of $\Delta'$, one can apply Lemma \ref{lemma_convex_1} to find a unique face $\F$ of $\De$ cut by $N^G$ in its interior such that $\F'=\F \cap N^G$. Lemma \ref{lemma_convex_2} says that $V(\F)$ splits in at least $k$ orbits.
Since $F$ is fibre-like, $V(\F)$ splits in exactly $k$ orbits: another orbit is given by the set of vertices $V(\De)\backslash V(\F)$.

Let now $\F'_1$ and $\F'_2$ be two distinct facets of $\De'$, which correspond to two faces $\F_1$ and $\F_2$ of $\De$ and determine two sets $S_1$ and $S_2$ of $k$ orbits in $V(\De)$. We claim that $S_1 \neq S_2$: otherwise 
$$V(\F_1) = \bigcup S_1 = \bigcup S_2 = V(\F_2),$$ 
which would imply $\F_1 = \F_2$.

Since there are exactly $k+1$ ways to choose $k$ elements in a set of cardinality $k+1$ and $\De'$ has at least $k+1$ facets (actually exactly $k+1$ by the above argument), we conclude that any collection of $k$ orbits is supported on a face $\F$ of $\De$.

Let $\mathcal{P}$ be a primitive collection with trivial focus, whose existence is guaranteed by Proposition \ref{batyrev}. Since any set of $k$ orbits must be contained in a face, $\mathcal{P}$ must involve at least one vertex from every orbit. Acting with $G$ on $\mathcal{P}$ we obtain a  family of primitive collections $\{ \mathcal{P}_i \}_{1\le i \le r }$ such that $\sigma(\mathcal{P}_i)=0$ and $\cup \mathcal{P}_i=V(\De)$. Assume that $\mathcal{P}_i \cap \mathcal{P}_j \ne \emptyset$ for some $i,j$, i.e., $\mathcal{P}_i=\{x_1, \ldots, x_k\}$ and $\mathcal{P}_j=\{x_1, \ldots, x_h, y_{h+1}, \ldots ,y_k\}$ with $y_s \ne x_t$ for any $s,t$. Then 
$$
x_{h+1}+\ldots +{x_k}=y_{h+1}+\ldots y_k,
$$ 
which is impossible by Remark \ref{rmk}, because $x_{h+1}, \ldots, x_k$ generate a cone in $\Sigma$.

This implies that all the $\mathcal{P}_i$ are disjoint and, as a consequence, that the sum of all vertices of $\De$ equals the origin, i.e. the barycentre of $\De$ is trivial. The theorem is proved.
\endproof

The previous corollary seems to be the relative version in the toric case of the following very general conjecture by Odaka and Okada.

\begin{con}[{\cite[Conj. 5.1]{OO13}}]
Any smooth Fano manifold $X$ of Picard rank 1 is  $K$-semistable.
\end{con}

\subsection{MAGMA computations}

Theorem \ref{cor_kstability} and its proof show that the fibre-like condition is rather restrictive.

Table \ref{t1} collects the smooth Fano toric varieties (up to dimension 8) which are fibre-like. It has been obtained using the software MAGMA together with the Graded Ring Database \cite{GRDB} (for further details on the classification, cf. \cite{Obro}).

\begin{rem}\label{super_skansen}
In Table \ref{t1}, the IDs of the Fano polytopes are the ones introduced in \cite{GRDB}.
The varieties $V_d$ are known as Del Pezzo varieties (see \cite{VK85} for more details). There is no classical description for the varieties $W_1,W_2$ and $W_3$.
\end{rem}

\begin{table}[!]\label{t1}
\begin{center}
\begin{tabular}{*{4}{c}}
DImension &  \# Vertices & Description  & ID\\
\hline
 $2$  & $6$ & $V_2$  & $2$  \\
 2 &  4 & $\pp^1 \times \pp^1$  & 4  \\
  2 &  3 & $\pp^2$  & 5 \\
\hline
 3 & 6 & $(\pp^1)^3$  & 21   \\
 3 & 4 & $\pp^3$ &  23     \\
\hline
4 &  10 & $V_4$  & 63   \\
4 & 12 & $V_2 \times V_2$  & 100 \\
4 & 8 & $(\pp^1)^4$  & 142  \\
4 & 6 & $\pp^2 \times \pp^2$  & 146    \\
4 & 5 & $\pp^4$  & 147    \\
\hline
5 & 10 & $(\pp^1)^5$  & 1003    \\
5 & 6 & $\pp^5$  & 1013    \\
\hline
6 &  14 & $V_6$  & 1930   \\
6 & 12 & $W_1$  & 5817 \\
6 & 18 & $(V_2)^3$  & 7568  \\
6 & 12 & $(\pp^1)^6$  & 8611  \\
6 & 9 & $(\pp^2)^3$  & 8631    \\
6 & 8 & $(\pp^3)^2$  & 8634    \\
6 & 7 & $\pp^6$  & 8635    \\
\hline
7 & 14 & $(\pp^1)^7$  & 80835    \\
7 & 8 & $\pp^7$  & 80891   \\
\hline
8 & 18 & $V_8$  & 106303   \\
8 & 15 & $W_2$  & 277415  \\
8 & 20 & $(V_4)^2$  & 442179  \\
8 & 24 & $(V_2)^4$  & 790981  \\
8 & 12 & $W_3$  & 830429    \\
8 & 16 & $(\pp^1)^8$	 & 830635    \\
8 & 12 & $(\pp^2)^4$  & 830767    \\
8 & 10 & $(\pp^4)^2$  & 830782    \\
8 & 9 & $\pp^8$  & 830783    \\
\end{tabular}
\caption{Smooth toric Fano varieties of dimension at most $8$ that are fibre-like.}
\end{center}
\end{table}
We would like to finish off by stating the following speculation.

\begin{con}
Let $\De$ be a smooth fibre-like polytope of dimension $d$. Assume that $d$ is an odd prime number.  Then either $X(\De)=\pp^d$ or $X(\De)=(\pp^1)^d$.   
\end{con}

\begin{section}{Rational homogeneous spaces}\label{sec_bandiere}
In this section we classify fibre-like rational homogeneous spaces; the upshot is that most of them are not fibre-like. Let us fix the notation.
\begin{definition}
A homogeneous space is a projective variety $F$ endowed with a transitive action of an algebraic group $G$.
\end{definition}
We assume that $G$ is semi-simple. In this set up, $F$ is a rational Fano varieties. Because of this remark, we refer to these varieties as {\it rational homogeneous space}. Alternatively, $F$ can be defined as a quotient of $G$ by a parabolic subgroup $P$. General references are \cite{Brion} and \cite{Demazure}. 

The isomorphism class of $F$ is determined by the conjugacy class of $P$; conjugacy classes of parabolic subgroups are in bijective correspondence with subsets of the nodes of the Dynkin diagram of $G$. We picture them by marking the corresponding nodes of the diagram; the resulting decorated diagram is called the Dynkin diagram of $P$. We denote by $E_P$ the group of symmetries of the Dynkin diagram preserving the marked nodes; it is a finite group and it is isomorphic to $\HMon(F)$ (see the proof of Corollary \ref{cor_hom}).  We are going to use the following classical result (cf. \cite[Theorems 1 and 2]{Demazure}).

\begin{theorem}[Demazure]
Let $F=G/P$ be a rational homogeneous space of Picard number at least 2, with $G$ simple. Then $F$ is rigid, that is $h^1(F,T_F)=0$. Moreover, the automorphism group of $F$ is isomorphic to the semi-direct product of $G$ and the symmetry $E_P$ of the Dynkin diagram of $P$.
\end{theorem}
Let us make a few comments. The rational homogeneous spaces which are called exceptional in \cite{Demazure} have Picard number one, so we can ignore them. The group of exterior automorphisms of $G$, which is denoted by $E$ in \cite{Demazure}, is known to be equal to the symmetries of the Dynkin diagram, see e.g. \cite[Section 10.6.10]{Procesi}. The group that Demazure calls $E_{\pi}$ here is denoted by $E_P$. The following result follows directly from Theorem \ref{mainthm} and Demazure's result.

\begin{corollary}\label{cor_hom}
Let $F=G/P$ be a homogeneous space of Picard rank at least $2$, with $G$ simple. Then $F$ is fibre-like if and only if is it isomorphic to one of the following varieties:
\begin{enumerate}
\item $F=F(n,k)$ parametrises pairs of subspaces $(L,H)$ in $\C^n$ such that $\dim L=k$, $\dim H=n-k$ and $L\subset H$;
\item $F=F(n)$ parametrises $n$ dimensional isotropic subspaces of ($\C^{2n},Q$), where $Q$ is a non-degenrate symmetric form;
\item $F=F^{\tau}$ parametrises pairs of isotropic subspaces $(L,\Pi)$ in  ($\C^8,Q$), where $Q$ is a non-degenrate symmetric form, $L$ is a line, $\Pi$ is four-dimensional and $L\subset \Pi$ (the upper-script $\tau$ stands for triality);
\item $F=F^{\tau}_i$ is the target of a contraction of a facet of the nef cone of $F^{\tau}$; more explicitly either $F^{\tau}_i=F(4)$ or $\Pi$ is forced to belong to one of the two connected component of the grassmannians of isotropic $4$-dimensional subspaces of $\C^8$.
\item $F=G/P$, where $G$ is the exceptional group $E_6$ and $P$ is associated to one of the two pairs of roots conjugated by the automorphism of the Dynkin diagram.
\end{enumerate}
\end{corollary}
Before giving the proof, let us give some details about these varieties. Case $(1)$ is realised as an homogeneous space with $G=SL_n$; it has Picard number two and the faces of the nef cone are given by the projection onto grassmannians. If we fix a quadratic form $Q$, we get an automorphism of $F$ given by $\phi_Q(L,H)=(H^{\perp},L^{\perp})$ which exchanges the faces of the nef cone. In Case $(2)$, $G=SO_{2n}$. The linear subspaces of an even dimensional quadric are divided into two families, which correspond to the even and odd spin representations (cf. \cite[Section 11.7.2]{Procesi} or \cite[Section 6.1]{GH}). This variety has Picard number two; the action of an improper orthogonal transformation exchanges the faces of the nef cone. The third variety is homogeneous for $G=SO_8$; it has Picard number $3$. The relevant automorphisms are realised via triality (e.g. \cite[Section 11.7.3]{Procesi}).

\begin{proof}
Since $F$ is rigid, because of Theorem \ref{mainthm} we have just to study the action of $\Aut(F)= G\rtimes E_P$ on $N^1(F)_{\qq}$. The group $G$ acts trivially in cohomology, so we are left with the action of the finite group $E_P$. The group $N^1(F)_{\qq}$ is spanned by the line bundles associated to the simple roots of $P$ (cf. \cite{Brion}), so we can identify a basis of $N^1(F)_{\qq}$ with the set of the marked nodes of the Dynkin diagram of $P$. This identification is equivariant with respect to the group of symmetry $E_P$; in particular $\HMon(F)$ equals to $E_P$. In other words, $\dim N^1(F)_{\qq}^{\Aut(F)}=1$ if and only if the group of symmetry $E_P$ acts transitively on the set of marked nodes. Dynkin diagram and their symmetries are classified (e.g. \cite[Section 10.6.10]{Procesi}); by direct inspection, we conclude that the unique $F$ which are fibre-like are the ones listed above.  More explicitly,  the Dynkin diagrams $B_n$, $C_n$, $E_7$, $E_8$, $F_4$ and $G_2$ have no symmetries, so the rational homogeneous spaces for the respective groups are fibre-like if and only if the Picard number is one. $A_n$ has just an order two symmetry, so for each pair of conjugated nodes one gets a fibre-like homogeneous space of Picard number $2$; this is case $(1)$. $D_n$, for $n\geq 4$, has just an order-two symmetry which fixes all nodes except the two nodes associated to the Spin representations, this gives case $(2)$. $D_4$ has the symmetric group on three elements as group of symmetries, this is the so called triality and gives cases $(3)$ and $(4)$. The Dynkin diagram $E_6$ has an order-two symmetry which gives case $(5)$.
\end{proof}

\end{section}

\bibliographystyle{alpha}
\bibliography{./bibliografia}

\end{document}